\documentclass[10pt]{article}
\usepackage{amsmath,amsfonts,amssymb,amsthm,graphicx}
\usepackage{enumerate}
\usepackage{color}
\usepackage{natbib}
\setcitestyle{square}
\usepackage{algorithm}
\usepackage[noend]{algpseudocode}  
\usepackage[colorlinks=true,citecolor=blue,pdfpagemode=UseNone,pdfstartview=FitH]{hyperref}

\renewcommand{\d}{\,\mathrm{d}}
\newcommand{\dd}{\mathrm{d}}
           
\newcommand{\E}{\mathbb{E}}    
\newcommand{\R}{\mathbb{R}}    
\newcommand{\N}{\mathbb{N}}    

\newcommand{\id}{1}            

\newcommand{\AAA}{\mathcal{A}}   
\newcommand{\EEE}{\mathcal{E}}   
\newcommand{\iEEE}{\mathit{i}\mathcal{E}}   
\newcommand{\FFF}{\mathcal{F}}   
\newcommand{\MMM}{\mathcal{M}}   
\newcommand{\PPP}{\mathcal{P}}   
\newcommand{\QQQ}{\mathcal{Q}}   
\newcommand{\UUU}{\mathcal{U}}   

\DeclareMathOperator*{\VS}{VS}   

\DeclareMathOperator*{\fdr}{fdr}
\renewcommand{\complement}{\textsf{c}}  

\theoremstyle{plain}
\newtheorem{theorem}{Theorem}[section]
\newtheorem{corollary}[theorem]{Corollary}
\newtheorem{lemma}[theorem]{Lemma}
\newtheorem{proposition}[theorem]{Proposition}

\theoremstyle{definition}

\theoremstyle{remark}
\newtheorem{remark}[theorem]{Remark}

\begin{document}
  \title{E-values: Calibration, combination, and applications}
  \author{Vladimir Vovk\thanks%
    {Department of Computer Science,
    Royal Holloway, University of London,
    Egham, Surrey, UK.
    E-mail: \href{mailto:v.vovk@rhul.ac.uk}{v.vovk@rhul.ac.uk}.}
  \and
    Ruodu Wang\thanks%
    {Department of Statistics and Actuarial Science,
    University of Waterloo,
    Waterloo, Ontario, Canada.
    E-mail: \href{mailto:wang@uwaterloo.ca}{wang@uwaterloo.ca}.}}
  \maketitle

\begin{abstract}
  Multiple testing of a single hypothesis and testing multiple hypotheses
  are usually done in terms of p-values.
  In this paper we replace p-values with their natural competitor, e-values,
  which are closely related to betting, Bayes factors, and likelihood ratios.
  We demonstrate that e-values are often mathematically more tractable;
  in particular, in multiple testing of a single hypothesis,
  e-values can be merged simply by averaging them.
  This allows us to develop efficient procedures using e-values
  for testing multiple hypotheses.
\end{abstract}

\section{Introduction}

The problem of multiple testing of a single hypothesis
(also known as testing a global null)
is usually formalized
as that of combining a set of p-values.
The notion of p-values, however, has a strong competitor, which we refer to as e-values in this paper.
E-values can be traced back to various old ideas,
but they have started being widely discussed in their pure form only recently:
see, e.g., \citet{Shafer:2019},
who uses the term ``betting score'' in the sense very similar to our ``e-value'',
\citet[Section 11.5]{Shafer/Vovk:2019}, who use ``Skeptic's capital'',
and \citet{Grunwald/etal:2019}.
The power and intuitive appeal of e-values stem from their interpretation
as results of bets against the null hypothesis \citep[Section~1]{Shafer:2019}.

Formally, an \emph{e-variable} is a nonnegative extended random variable
whose expected value under the null hypothesis is at most 1,
and an \emph{e-value} is a value taken by an e-variable.
Whereas p-values are defined in terms of probabilities, 
e-values are defined in terms of expectations. 
As we regard an e-variable $E$ as a bet against the null hypothesis,
its realized value $e:=E(\omega)$ shows how successful our bet is
(it is successful if it multiplies the money it risks by a large factor).
Under the null hypothesis, it can be larger than a constant $c>1$ with probability at most $1/c$
(by Markov's inequality).
If we are very successful (i.e., $e$ is very large),
we have reasons to doubt that the null hypothesis is true,
and $e$ can be interpreted as the amount of evidence we have found against it.
In textbook statistics e-variables typically appear under the guise of likelihood ratios and Bayes factors.

The main focus of this paper is on combining e-values
and multiple hypothesis testing using e-values.
The picture that arises for these two fields is remarkably different from,
and much simpler than,
its counterpart for p-values.
To clarify connections between e-values and p-values,
we discuss how to transform p-values into e-values,
or \emph{calibrate} them,
and how to move in the opposite direction.

We start the main part of the paper by defining the notion of e-values in Section~\ref{sec:p-vs-e}
and reviewing known results about
connections between e-values and p-values;
we will discuss how the former can be turned into the latter and vice versa
(with very different domination structures for the two directions).
In Section~\ref{sec:e} we show that the problem of merging e-values is more or less trivial:
a convex mixture of e-values is an e-value,
and symmetric merging functions are essentially dominated by the arithmetic mean.
For example, when several analyses are conducted on a common (e.g., public) dataset each reporting an e-value,
it is natural to summarize them as a single e-value equal to their weighted average (the same cannot be said for p-values).
In Section~\ref{sec:ie} we assume, additionally,
that the e-variables being merged are independent and show that the domination structure is much richer;
for example, now the product of e-values is an e-value.
The assumption of independence can be replaced by the weaker assumption of being \emph{sequential},
and we discuss connections with the popular topic of using martingales in statistical hypothesis testing:
see, e.g., \citet{Duan/etal:2019} and \citet{Shafer/Vovk:2019}.
In Section~\ref{sec:multiple} we apply these results to multiple hypothesis testing.
In the next section, Section~\ref{sec:p}, we briefly review known results on merging p-values
(e.g., the two classes of merging methods in \citet{Ruger:1978} and \citet{Vovk/Wang:2019})
and draw parallels with merging e-values;
in the last subsection we discuss the case where p-values are independent.
Section~\ref{sec:experiments} is devoted to experimental results;
one finding in this section is that,
for multiple testing of a single hypothesis in independent experiments,
a simple method based on e-values outperforms standard methods based on p-values.
Section~\ref{sec:conclusion} concludes the main part of the paper.

Appendix~\ref{app:survey} describes numerous connections with the existing literature,
including Bayes factors and multiple hypothesis testing.
Appendix~\ref{app:calibrators} describes the origins of the problem of calibrating p-values
and gives interesting  examples of calibrators.
A short Appendix~\ref{app:infty} deals with merging infinite e-values.
Appendix~\ref{app:atomless} explores the foundations of calibration and merging of e-values and p-values;
in particular, whether the universal quantifiers over probability spaces in the definitions given in the main paper are really necessary.
Appendix~\ref{app:domination} proves Theorem~\ref{thm:iff} in the main paper
characterizing the domination structure of the e-merging functions.
Appendix~\ref{app:minimax} presents an informative minimax view of essential and weak domination.
Appendix~\ref{app:cross-merging} discusses ``cross-merging'':
how do we merge several p-values into one e-value and several e-values into one p-value?
Appendix~\ref{app:extra-experiments} contains additional experimental results.
Finally, Appendix~\ref{app:FACT} briefly describes the procedure that we use for multiple hypothesis testing
in combination with Fisher's [\citeyear{Fisher:1932}] method of combining p-values.

The journal version of this paper is to appear in the \emph{Annals of Statistics}.

\section{Definition of e-values and connections with p-values}
\label{sec:p-vs-e}

For a probability space $(\Omega,\AAA,Q)$,
an \emph{e-variable} is an extended random variable $E:\Omega\to[0,\infty]$ satisfying
$\E^Q[E] \le 1$;
we refer to it as ``extended'' since its values are allowed to be $\infty$,
and we let $\E^Q[X]$ (or $\E[X]$ when $Q$ is clear from context)
stand for $\int X\d Q$ for any extended random variable $X$.
The values taken by e-variables will be referred to as \emph{e-values},
and we denote the set of e-variables by $\EEE_Q$.
It is important to allow $E$ to take value $\infty$;
in the context of testing $Q$,
observing $E=\infty$ for an \emph{a priori} chosen e-variable $E$
means that we are entitled to reject $Q$ as null hypothesis.

Our emphasis in this paper is on e-values,
but we start from discussing their connections with the familiar notion of p-values.
A \emph{p-variable} is a random variable $P:\Omega\to[0,1]$ satisfying
\[
  \forall\epsilon\in(0,1):
  Q(P\le\epsilon)
  \le
  \epsilon.
\]
The set of all p-variables is denoted by $\PPP_Q$.

A calibrator is a function transforming p-values to e-values.
Formally, a decreasing function $f:[0,1]\to[0,\infty]$ is a \emph{calibrator}
(or, more fully, \emph{p-to-e calibrator})
if, for any probability space $(\Omega,\AAA,Q)$ and any p-variable $P\in\PPP_Q$, $f(P)\in\EEE_Q$.
A calibrator $f$ is said to \emph{dominate} a calibrator $g$ if $f\ge g$,
and the domination is \emph{strict} if $f\ne g$.
A calibrator is \emph{admissible} if it is not strictly dominated by any other calibrator.

The following proposition says that a calibrator
is a nonnegative decreasing function integrating to at most 1
over the uniform probability measure.

\begin{proposition}\label{prop:p-to-e}
  A decreasing function $f:[0,1]\to[0,\infty]$ is a calibrator if and only if $\int_0^1 f \le 1$.
  It is admissible if and only if $f$ is upper semicontinuous, $f(0)=\infty$, and $\int_0^1 f = 1$.
\end{proposition}

Of course, in the context of this proposition,
being upper semicontinuous is equivalent to being left-continuous.

\begin{proof}
  Proofs of similar statements are given in, e.g., \citet[Theorem 7]{Vovk:1993},
  \citet[Theorem~3]{Shafer/etal:2011}, and \citet[Proposition~11.7]{Shafer/Vovk:2019},
  but we will give an independent short proof using our definitions.
  The first ``only if" statement is obvious.
  To show the first ``if" statement,
  suppose that $\int_0^1 f \le 1$, $P$ is a p-variable,
  and $P'$ is uniformly distributed on $[0,1]$.
  Since $Q(P<x) \le Q(P'<x)$ for all $x\ge 0$ and $f$ is decreasing, 
  we have 
  \[
    Q(f(P)>y) \le Q(f(P')>y)
  \]
  for all $y\ge0$, which implies
  \[
    \E[f(P)]
    \le
    \E[f(P')]
    =
    \int_0^1 f(p) \d p \le 1.
  \]
  The second statement in Proposition~\ref{prop:p-to-e} is obvious.
\end{proof}

The following is a simple family of calibrators.
Since $\int_0^1 \kappa p^{\kappa-1} \d p = 1$,
the functions
\begin{equation}\label{eq:calibrator}
  f_{\kappa}(p)
  :=
  \kappa p^{\kappa-1}
\end{equation}
are calibrators,
where $\kappa\in(0,1)$.
To solve the problem of choosing the parameter $\kappa$,
sometimes the maximum
\[
  \VS(p)
  :=
  \max_{\kappa\in[0,1]}
  f_{\kappa}(p)
  =
  \begin{cases}
    -\exp(-1)/(p\ln p) & \text{if $p\le\exp(-1)$}\\
    1 & \text{otherwise}
  \end{cases}
\]
is used (see, e.g., \citet{Benjamin/Berger:2019}, Recommendations 2 and 3);
we will refer to it as the \emph{VS bound}
(abbreviating ``Vovk--Sellke bound'', as used in, e.g., the JASP package).
It is important to remember that $\VS(p)$ is not a valid e-value,
but just an overoptimistic upper bound on what is achievable
with the class \eqref{eq:calibrator}.
Another way to get rid of $\kappa$ is to integrate over it,
which gives
\begin{equation}\label{eq:reviewer}
  F(p)
  :=
  \int_0^1 \kappa p^{\kappa-1} \d\kappa
  =
  \frac{1-p+p\ln p}{p(-\ln p)^2}.
\end{equation}
(See Appendix~\ref{app:calibrators} for more general results and references.
We are grateful to Aaditya Ramdas for pointing out the calibrator \eqref{eq:reviewer}.)
An advantage of this method is that it produces a bona fide e-value,
unlike the VS bound. 
As $p\to0$, $F(p)\sim p^{-1}(-\ln p)^{-2}$,
so that $F(p)$ is closer to the ideal (but unachievable) $1/p$
(cf.\ Remark~\ref{rem:rough} below)
than any of \eqref{eq:calibrator}.

In the opposite direction, an e-to-p calibrator is a function transforming e-values to p-values.
Formally, a decreasing function $f:[0,\infty]\to[0,1]$ is an \emph{e-to-p calibrator}
if, for any probability space $(\Omega,\AAA,Q)$ and any e-variable $E\in\EEE_Q$, $f(E)\in\PPP_Q$.
The following proposition,
which is the analogue of Proposition~\ref{prop:p-to-e} for e-to-p calibrators,
says that there is, essentially, only one e-to-p calibrator, $f(t):=\min(1,1/t)$.

\begin{proposition}\label{prop:e-to-p}
  The function $f:[0,\infty]\to[0,1]$ defined by $f(t):=\min(1,1/t)$ is an e-to-p calibrator.
  It dominates every other e-to-p calibrator.
  In particular, it is the only admissible e-to-p calibrator.
\end{proposition}

\begin{proof}
  The fact that $f(t):=\min(1,1/t)$ is an e-to-p calibrator follows from Markov's inequality:
  if $E\in\EEE_Q$ and $\epsilon\in(0,1)$,
  \[
    Q(f(E)\le\epsilon)
    =
    Q(E\ge 1/\epsilon)
    \le
    \frac{\E^Q[E]}{1/\epsilon}
    \le
    \epsilon.
  \]

  On the other hand, suppose that $f$ is another e-to-p calibrator.
  It suffices to check that $f$ is dominated by $\min(1,1/t)$.
  Suppose $f(t)<\min(1,1/t)$ for some $t\in[0,\infty]$.
  Consider two cases:
  \begin{itemize}
  \item
    If $f(t)<\min(1,1/t)=1/t$ for some $t>1$,
    fix such $t$ and consider an e-variable $E$ that is $t$ with probability $1/t$ and $0$ otherwise.
    Then $f(E)$ is $f(t)<1/t$ with probability $1/t$,
    whereas it would have satisfied $P(f(E)\le f(t))\le f(t)<1/t$ had it been a p-variable.
  \item
    If $f(t)<\min(1,1/t)=1$ for some $t\in[0,1]$,
    fix such $t$ and consider an e-variable $E$ that is $1$ a.s.
    Then $f(E)$ is $f(t)<1$ a.s.,
    and so it is not a p-variable.
    \qedhere
  \end{itemize}
\end{proof}

Proposition~\ref{prop:p-to-e} implies that the domination structure of calibrators is very rich,
whereas Proposition~\ref{prop:e-to-p} implies that the domination structure of e-to-p calibrators is trivial.

\begin{remark}\label{rem:rough}
  A possible interpretation of this section's results
  is that e-variables and p-variables are connected via a rough relation $1/e \sim p$.
  In one direction, the statement is precise:
  the reciprocal (truncated to 1 if needed) of an e-variable is a p-variable by Proposition~\ref{prop:e-to-p}.
  On the other hand,
  using a calibrator \eqref{eq:calibrator} with a small $\kappa>0$
  and ignoring positive constant factors
  (as customary in the algorithmic theory of randomness,
  discussed in Section~\ref{subsec:ATR}),
  we can see that the reciprocal of a p-variable is approximately an e-variable.
  In fact, $f(p)\le 1/p$ for all $p$ when $f$ is a calibrator;
  this follows from Proposition~\ref{prop:p-to-e}.
  However, $f(p)=1/p$ is only possible in the extreme case $f=\id_{[0,p]}/p$.
\end{remark}

\section{Merging e-values}
\label{sec:e}

An important advantage of e-values over p-values is that they are easy to combine.
This is the topic of this section, in which we consider the general case,
without any assumptions on the joint distribution of the input e-variables.
The case of independent e-variables is considered in the next section.

Let $K\ge2$ be a positive integer (fixed throughout the paper apart from Section~\ref{sec:experiments}).
An \emph{e-merging function} of $K$ e-values is an increasing Borel function
$F:[0,\infty]^K\to[0,\infty]$
such that, for any probability space $(\Omega,\AAA,Q)$ and random variables $E_1,\dots,E_K$ on it,
\begin{equation}\label{eq:def-e}
  E_1,\dots,E_K \in \EEE_Q
  \Longrightarrow
  F(E_1,\dots,E_K) \in \EEE_Q
\end{equation}
(in other words, $F$ transforms e-values into an e-value).
In this paper we will also refer to increasing Borel functions $F:[0,\infty)^K\to[0,\infty)$
satisfying \eqref{eq:def-e} for all probability spaces and all e-variables $E_1,\dots,E_K$ taking values in $[0,\infty)$
as e-merging functions;
such functions are canonically extended to e-merging functions $F:[0,\infty]^K\to[0,\infty]$
by setting them to $\infty$ on $[0,\infty]^K\setminus[0,\infty)^K$
(see Proposition \ref{prop:infty} in Appendix~\ref{app:infty}).

An e-merging function $F$ \emph{dominates} an e-merging function $G$ if $F\ge G$
(i.e., $F(\mathbf{e})\ge G(\mathbf{e})$ for all $\mathbf{e}\in[0,\infty)^K$).
The domination is \emph{strict} (and we say that $F$ \emph{strictly dominates} $G$)
if $F\ge G$ and $F(\mathbf{e})>G(\mathbf{e})$ for some $\mathbf{e}\in[0,\infty)^K$.
We say that an e-merging function $F$ is \emph{admissible}
if it is not strictly dominated by any e-merging function;
in other words, admissibility means being maximal in the partial order of domination.

A fundamental fact about admissibility is proved in Appendix~\ref{app:domination}
(Proposition~\ref{prop:dominated}):
any e-merging function is dominated by an admissible e-merging function.

\subsection*{Merging e-values via averaging}

In this paper we are mostly interested in symmetric merging functions
(i.e., those invariant w.r.\ to permutations of their arguments).
The main message of this section is that the most useful
(and the only useful, in a natural sense)
symmetric e-merging function is the \emph{arithmetic mean}
\begin{equation}\label{eq:average}
  M_K(e_1,\dots,e_K)
  :=
  \frac{e_1+\dots+e_K}{K},
  \qquad
  e_1,\dots,e_K \in [0,\infty).
\end{equation}
In Theorem \ref{thm:iff} below we will see that $M_K$ is admissible
(this is also a consequence of Proposition~\ref{prop:IPK}).
But first we state formally the vague claim that $M_K$ is the only useful symmetric e-merging function.

An e-merging function $F$ \emph{essentially dominates} an e-merging function $G$ if,
for all $\textbf{e}\in[0,\infty)^K$,
\[
  G(\mathbf{e}) > 1
  \Longrightarrow
  F(\mathbf{e})\ge G(\mathbf{e}).
\]
This weakens the notion of domination in a natural way:
now we require that $F$ is not worse than $G$ only in cases where $G$ is not useless;
we are not trying to compare degrees of uselessness.
The following proposition can be interpreted as saying
that $M_K$ is at least as good as any other symmetric e-merging function.

\begin{proposition}\label{prop:M}
 The arithmetic mean $M_K$ essentially dominates any symmetric e-merging function.
\end{proposition}

In particular, if $F$ is an e-merging function that is symmetric
and positively homogeneous
(i.e., $F(\lambda\mathbf{e})=\lambda F(\mathbf{e})$ for all $\lambda>0$),
then $F$ is dominated by $M_K$.
This includes the e-merging functions discussed later in Section~\ref{sec:p}.

\begin{proof}[Proof of Proposition~\ref{prop:M}]
  Let $F$ be a symmetric e-merging function.
  Suppose for the purpose of contradiction that there exists $(e_1,\dots,e_K)\in[0,\infty)^K$ such that
  \begin{equation}\label{eq:conj1}
    b:=F(e_1,\dots,e_K)
    > 
    \max\left(\frac{e_1+\dots+e_K}{K},1
    \right)=:a.
  \end{equation} 
  Let $\Pi_K$ be the set of all permutations of $\{1,\dots,K\}$, 
  $\pi$ be randomly and uniformly drawn from $\Pi_K$, and $(D_1,\dots,D_K):=(e_{\pi(1)},\dots,e_{\pi(K)})$. 
  Further, let $(D'_1,\dots,D'_K):=(D_1,\dots,D_K)\id_A$,
  where $A$ is an event independent of $\pi$ and satisfying $P(A) = 1/a$
  (the existence of such random $\pi$ and $A$ is guaranteed for any atomless probability space
  by Lemma~\ref{lem:rich} in Appendix~\ref{app:atomless}).
  
  For each $k$, since $D_k$ takes the values $e_1,\dots,e_K$ with equal probability,
  we have $\E[D_k] = (e_1+\dots+e_K)/K$, which implies $\E[D'_k] = (e_1+\dots+e_K)/(Ka) \le 1$.
  Together with the fact that $D'_k$ is nonnegative, we know $D'_k\in\EEE_Q$.
  Moreover, by symmetry,
  \[
    \E[F(D'_1,\dots,D'_K)]
    =
    Q(A)
    F(e_1,\dots,e_K)
    +
    (1-Q(A))
    F(0,\dots,0)
    \ge
    b/a > 1,
  \]
  a contradiction.
  Therefore, we conclude that there is no $(e_1,\dots,e_K)$ such that \eqref{eq:conj1} holds.
%
\end{proof}

It is clear that the arithmetic mean $M_K$ does not dominate every symmetric e-merging function;
for example, the convex mixtures
\begin{equation}\label{eq:convex}
  \lambda + (1-\lambda) M_K,
  \quad
  \lambda\in[0,1],
\end{equation}
of the trivial e-merging function $1$ and $M_K$
are pairwise non-comparable (with respect to the relation of domination).
In the theorem below, we show that each of these mixtures is admissible
and that the class \eqref{eq:convex} is,
in the terminology of statistical decision theory \citep[Section 1.3]{Wald:1950},
a complete class of symmetric e-merging functions:
every symmetric e-merging function is dominated by one of \eqref{eq:convex}.
In other words, \eqref{eq:convex} is the minimal complete class of symmetric e-merging functions.

\begin{theorem}\label{thm:iff}
  Suppose that $F$ is a symmetric e-merging function.
  Then $F$ is dominated by the function $\lambda + (1-\lambda) M_K$ for some $\lambda\in[0,1]$.
  In particular, $F$ is admissible if and only if $F = \lambda + (1-\lambda) M_K$,
  where $\lambda = F(\mathbf{0})\in[0,1]$.
\end{theorem}
The proof of Theorem \ref{thm:iff} is put in Appendix~\ref{app:domination}
as it requires several other technical results in the appendix.
Finally, we note that, for $\lambda \ne 1$,
the functions in the class \eqref{eq:convex} carry the same statistical information.

\section{Merging independent e-values}
\label{sec:ie}

In this section we consider merging functions for independent e-values.
An \emph{ie-merging function} of $K$ e-values is an increasing Borel function
$F:[0,\infty)^K\to[0,\infty)$
such that $F(E_1,\dots,E_K)\in\EEE_Q$ for all independent $E_1,\dots,E_K\in\EEE_Q$
in any probability space $(\Omega,\AAA,Q)$.
As for e-merging functions,
this definition is essentially equivalent to the definition involving $[0,\infty]$ rather than $[0,\infty)$
(by Proposition~\ref{prop:infty} in Appendix~\ref{app:infty},
which is still applicable in the context of merging independent e-values).
The definitions of domination, strict domination, and admissibility 
are obtained from the definitions of the previous section by replacing ``e-merging'' with ``ie-merging''.

Let $\iEEE_Q^K\subseteq\EEE_Q^K$ be the set of (component-wise) independent random vectors in $\EEE_Q^K$,
and $\mathbf{1}:=(1,\dots,1)$ be the all-1 vector in $\R^K$.
The following proposition has already been used in Section \ref{sec:e} 
(in particular, it implies that the arithmetic mean $M_K$ is an admissible e-merging function).

\begin{proposition}\label{prop:IPK}
  For an increasing Borel function $F:[0,\infty)^K\to [0,\infty)$,
  if $\E[F(\mathbf{E})]=1$ for all $\mathbf{E}\in\EEE^K_Q$ with $\E[\mathbf{E}]=\mathbf{1}$
  (resp., for all $\mathbf{E}\in\iEEE^K_Q$ with $\E[\mathbf{E}]=\mathbf{1}$),
  then $F$ is an admissible e-merging function
  (resp., an admissible ie-merging function).
\end{proposition}

\begin{proof}
  It is obvious that $F$ is an e-merging function (resp., ie-merging function).
  Next we show that $F$ is admissible.
  Suppose for the purpose of contradiction that there exists an ie-merging function $G$ such that $G\ge F$ and 
  $
    G(e_1,\dots,e_K)>F(e_1,\dots,e_K)
  $
  for some $(e_1,\dots,e_K)\in [0,\infty)^K$.
  Take $(E_1,\dots,E_K)\in\iEEE^K_Q$ with $\E[(E_1,\dots,E_K)]=\mathbf{1}$
  such that $Q((E_1,\dots,E_K)=(e_1,\dots,e_K))>0$.
  Such a random vector is easy to construct by considering any distribution with a positive mass on each of $e_1,\dots,e_K$.
  Then we have
  \[
    Q(G(E_1,\dots,E_K) > F(E_1,\dots,E_K)) > 0,
  \]
  which implies
  \[
    \E[G(E_1,\dots,E_K)] > \E[G(E_1,\dots,E_K)] = 1,
  \]
  contradicting the assumption that $G$ is an ie-merging function.
  Therefore, no ie-merging function strictly dominates $F$.
  Noting that an e-merging function is also an ie-merging function,
  admissibility of $F$ is guaranteed under both settings.
\end{proof}

If $E_1,\dots,E_K$ are independent e-variables, their product $E_1\dots E_K$ will also be an e-variable.
This is the analogue of Fisher's [\citeyear{Fisher:1932}] method for p-values
(according to the rough relation $e\sim1/p$ mentioned in Remark~\ref{rem:rough};
Fisher's method is discussed at the end of Section~\ref{sec:p}).
The ie-merging function
\begin{equation}\label{eq:product}
  (e_1,\dots,e_K)
  \mapsto
  e_1\dots e_K
\end{equation}
is admissible by Proposition \ref{prop:IPK}.
It will be referred to as the \emph{product} (or \emph{multiplication})
ie-merging function.
The betting interpretation of \eqref{eq:product} is obvious:
it is the result of $K$ successive bets using the e-variables $E_1,\dots,E_K$
(starting with initial capital 1 and betting the full current capital $E_1\dots E_{k-1}$ on each $E_k$).

More generally, we can see that the U-statistics
\begin{equation}\label{eq:U}
  U_n(e_1,\dots,e_K)
  :=
  \frac{1}{\binom{K}{n}}
  \sum_{\{k_1,\dots,k_n\}\subseteq\{1,\dots,K\}}
  e_{k_1} \dots e_{k_n},
  \quad
  n\in\{0,1,\dots,K\},
\end{equation}
and their convex mixtures are ie-merging functions.
Notice that this class includes product (for $n=K$),
arithmetic average $M_K$ (for $n=1$),
and constant 1 (for $n=0$).
Proposition \ref{prop:IPK} implies that the U-statistics~\eqref{eq:U}
and their convex mixtures are admissible ie-merging functions.

The betting interpretation of a U-statistic~\eqref{eq:U}
or a convex mixture of U-statistics is implied by the betting interpretation
of each component $e_{k_1} \dots e_{k_n}$.
Assuming that $k_1,\dots,k_n$ are sorted in the increasing order,
$e_{k_1} \dots e_{k_n}$ is the result of $n$ successive bets using the e-variables $E_{k_1},\dots,E_{k_n}$;
and a convex mixture of bets corresponds to investing the appropriate fractions of the initial capital into those bets.

Let us now establish a very weak counterpart of Proposition~\ref{prop:M}
for independent e-values
(on the positive side it will not require the assumption of symmetry).
An ie-merging function $F$ \emph{weakly dominates} an ie-merging function $G$ if,
for all $e_1,\dots,e_K$,
\[
  (e_1,\dots,e_K)\in[1,\infty)^K
  \Longrightarrow
  F(e_1,\dots,e_K)\ge G(e_1,\dots,e_K).
\]
In other words, we require that $F$ is not worse than $G$
if all input e-values are useful
(and this requirement is weak because, especially for a large $K$,
we are also interested in the case where some of the input e-values are useless).

\begin{proposition}  \label{prop:M-i}
  The product $(e_1,\dots,e_K)\mapsto e_1\dots e_K$ weakly dominates any ie-merging function.
\end{proposition}

\begin{proof}
  Indeed, suppose that there exists $(e_1,\dots,e_K)\in[1,\infty)^K$ such that
  \[F(e_1,\dots,e_K)>e_1\dots e_K.\]
  Let $E_1,\dots,E_K$ be independent random variables
  such that each $E_k$ for $k\in\{1,\dots,K\}$ takes values in the two-element set $\{0,e_k\}$
  and $E_k=e_k$ with probability $1/e_k$.
  Then each $E_k$ is an e-variable but
  \begin{align*}
    \E[F(E_1,\dots,E_K)]
    &\ge
    F(e_1,\dots,e_K)
    Q(E_1=e_1,\dots,E_K=e_K)\\
    &>
    e_1\dots e_K
    (1/e_1)\dots(1/e_K)
    =
    1,
  \end{align*}
  which contradicts $F$ being an ie-merging function.
\end{proof}

\begin{remark} 
  A natural question is whether the convex mixtures of \eqref{eq:U} form a complete class.
  They do not:
  Proposition~\ref{prop:IPK} implies that
  \[
    f(e_1,e_2)
    :=
    \frac12
    \left(
      \frac{e_1}{1 + e_1}
      +
      \frac{e_2}{1 + e_2}
    \right)
    \left(
      1 + e_1 e_2
    \right)
  \]
  is an admissible ie-merging function,
  and it is easy to check that it is different from any convex mixture of \eqref{eq:U}.
\end{remark}

\subsection*{Testing with martingales}

The assumption of the independence of e-variables $E_1,\dots,E_K$
is not necessary for the product $E_1\dots E_K$ to be an e-variable.
Below, we say that the e-variables $E_1,\dots,E_K$ are \emph{sequential}
if  $\E[E_k\mid E_1,\dots,E_{k-1}]\le1$ almost surely for all $k\in\{1,\dots,K\}$.
Equivalently, the sequence of the partial products $(E_1\dots E_k)_{k=0,1,\dots,K}$ is a supermartingale
in the filtration generated by $E_1,\dots,E_K$
(or a \emph{test supermartingale},
in the terminology of \citet{Shafer/etal:2011}, \citet{Howard/etal:arXiv1810}, and \citet{Grunwald/etal:2019},
meaning a nonnegative supermartingale with initial value 1).
A possible interpretation of this test supermartingale
is that the e-values $e_1,e_2,\dots$ are obtained by laboratories $1,2,\dots$ in this order,
and laboratory $k$ makes sure that its result $e_k$ is a valid e-value given the previous results $e_1,\dots,e_{k-1}$.
The test supermartingale is a \emph{test martingale}
if $\E[E_k\mid E_1,\dots,E_{k-1}]=1$ almost surely for all $k$
(intuitively, it is not wasteful).

It is straightforward to check that all convex mixtures of \eqref{eq:U} (including the product function)
produce a valid e-value from sequential e-values;
we will say that they are \emph{se-merging functions}.
On the other hand, independent e-variables are sequential,
and hence se-merging functions form a subset of ie-merging functions.
In the class of se-merging functions, the convex mixtures of \eqref{eq:U} are admissible,
as they are admissible in the larger class of ie-merging functions (by Proposition \ref{prop:IPK}).
For the same reason (and by Proposition \ref{prop:M-i}),
the product function in \eqref{eq:product} weakly dominates every other se-merging function.
This gives a (weak) theoretical justification for us to use the product function
as a canonical merging method in Sections \ref{sec:multiple} and \ref{sec:experiments} for e-values as long as they are sequential.
Finally, we note that it suffices for $E_1,\dots,E_K$ to be sequential in any order for these merging methods
(such as Algorithm \ref{alg:BH-i} in Section \ref{sec:multiple})
to be valid.

\section{Application to testing multiple hypotheses}
\label{sec:multiple}

As in \citet{Vovk/Wang:2019},
we will apply results for multiple testing of a single hypothesis
(combining e-values in the context of Sections~\ref{sec:e} and~\ref{sec:ie})
to testing multiple hypotheses.
As we explain in Appendix~\ref{app:survey} (Section~\ref{subsec:MHT}),
our algorithms just spell out the application of the closure principle
\citep{Marcus/etal:1976,Goeman/Solari:2011},
but our exposition in this section will be self-contained.

Let $(\Omega,\AAA)$ be our sample space (formally, a measurable space),
and $\mathfrak{P}(\Omega)$ be the family of all probability measures on it.
A \emph{composite null hypothesis} is a set $H\subseteq\mathfrak{P}(\Omega)$ of probability measures
on the sample space.
We say that $E$ is an \emph{e-variable} w.r.\ to a composite null hypothesis $H$
if $\E^Q[E_k]\le1$ for any $Q\in H_k$.

\begin{algorithm}[bt]
  \caption{Adjusting e-values for multiple hypothesis testing}
  \label{alg:BH}
  \begin{algorithmic}[1]  
    \Require
      A sequence of e-values $e_1,\dots,e_K$.
    \State Find a permutation $\pi:\{1,\dots,K\}\to\{1,\dots,K\}$ such that $e_{\pi(1)}\le\dots\le e_{\pi(K)}$.
    \State Set $e_{(k)}:=e_{\pi(k)}$, $k\in\{1,\dots,K\}$ (these are the \emph{order statistics}).
    \State $S_0:=0$
    \For{$i=1,\dots,K$}
      \State $S_i := S_{i-1} + e_{(i)}$
    \EndFor
    \For{$k=1,\dots,K$}
      \State $e^*_{\pi(k)}:=e_{\pi(k)}$
      \For{$i=1,\dots,k-1$}
        \State $e := \frac{e_{\pi(k)}+S_i}{i+1}$
        \If{$e < e^*_{\pi(k)}$}
          \State $e^*_{\pi(k)} := e$
        \EndIf
      \EndFor
    \EndFor
  \end{algorithmic}
\end{algorithm}

In multiple hypothesis testing
we are given a set of composite null hypotheses $H_k$, $k=1,\dots,K$.
Suppose that, for each $k$, we are also given an e-variable $E_k$ w.r.\ to $H_k$.
Our multiple testing procedure is presented as Algorithm~\ref{alg:BH}.
The procedure adjusts the e-values $e_1,\dots,e_K$,
perhaps obtained in $K$ experiments (not necessarily independent),
to new e-values $e^*_1,\dots,e^*_K$;
the adjustment is downward in that $e^*_k\le e_k$ for all $k$.
Applying the procedure to the e-values $e_1,\dots,e_K$ produced by the e-variables $E_1,\dots,E_K$,
we obtain extended random variables $E^*_1,\dots,E^*_K$ taking values $e^*_1,\dots,e^*_K$.
The output $E^*_1,\dots,E^*_K$ of Algorithm~\ref{alg:BH} satisfies a property of validity
which we will refer to as \emph{family-wise validity} (FWV);
in Section~\ref{subsec:MHT} we will explain its analogy with the standard family-wise error rate (FWER).

A \emph{conditional e-variable} is a family of extended nonnegative random variables $E_Q$,
$Q\in\mathfrak{P}(\Omega)$,
that satisfies
\[
  \forall Q\in\mathfrak{P}(\Omega):
  \E^Q[E_Q]
  \le
  1
\]
(i.e., each $E_Q$ is in $\EEE_Q$).
We regard it as a system of bets against each potential data-generating distribution $Q$.

Extended random variables $E^*_1,\dots,E^*_K$ taking values in $[0,\infty]$
are \emph{family-wise valid} (\emph{FWV}) for testing $H_1,\dots,H_K$
if there exists a conditional e-variable $(E_Q)_{Q\in\mathfrak{P}(\Omega)}$ such that
\begin{equation}\label{eq:FWV}
  \forall k\in\{1,\dots,K\} \;
  \forall Q\in H_k:
  E_Q \ge E^*_k
\end{equation}
(where $E_Q \ge E^*_k$ means, as usual, that $E_Q(\omega) \ge E^*_k(\omega)$ for all $\omega\in\Omega$).
We can say that such $(E_Q)_{Q\in\mathfrak{P}(\Omega)}$ \emph{witnesses} the FWV property
of $E^*_1,\dots,E^*_K$.

The interpretation of family-wise validity is based on our interpretation of e-values.
Suppose we observe an outcome $\omega\in\Omega$.
If $E_Q(\omega)$ is very large, we may reject $Q$ as the data-generating distribution.
Therefore, if $E^*_k(\omega)$ is very large, we may reject the whole of $H_k$ (i.e., each $Q\in H_k$).
In betting terms, we have made at least $\$E^*_k(\omega)$ risking at most \$1 when gambling against any $Q\in H_k$.

Notice that we can rewrite \eqref{eq:FWV} as
\begin{equation*}
  \forall Q\in\mathfrak{P}(\Omega):
  \E^Q
  \left[
    \max_{k
    : Q\in H_k}
    E^*_k
  \right]
  \le
  1.
\end{equation*}
In other words, we require joint validity of the e-variables $E^*_k$.

We first state the validity of Algorithm \ref{alg:BH}
(as well as Algorithm \ref{alg:BH-i} given below),
and our justification  follows.

\begin{theorem}
  Algorithms~\ref{alg:BH} and~\ref{alg:BH-i} are family-wise valid.
\end{theorem}

Let us check that the output $E^*_1,\dots,E^*_K$ of Algorithm~\ref{alg:BH} is FWV.
For $I\subseteq\{1,\dots,K\}$,
the composite hypothesis $H_I$ is defined by
\begin{equation}\label{eq:H}
  H_I
  :=
  \left(
    \bigcap_{k\in I} H_k
  \right)
  \bigcap
  \left(
    \bigcap_{k\in\{1,\dots,K\}\setminus I} H^\complement_k
  \right),
\end{equation}
where $H^\complement_k$ is the complement of $H_k$.
The conditional e-variable witnessing that $E^*_1,\dots,E^*_K$ are FWV 
is the arithmetic mean
\begin{equation}\label{eq:Q}
  E_Q
  :=
  \frac{1}{\left|I_Q\right|}
  \sum_{k\in I_Q}
  E_k,
\end{equation}
where $I_Q:=\{k\mid Q\in H_k\}$ and $E_Q$ is defined arbitrarily (say, as 1) when $I_Q=\emptyset$.
The optimal adjusted e-variables $E'_k$ can be defined as
\begin{equation}\label{eq:E-optimal}
  E'_k
  :=
  \min_{Q\in H_k}
  E_Q
  \ge
  \min_{I\subseteq\{1,\dots,K\}:k\in I}
  \frac{1}{\left|I\right|}
  \sum_{i\in I}
  E_i,
\end{equation}
but for computational efficiency we use the conservative definition
\begin{equation}\label{eq:E*}
  E^*_k
  :=
  \min_{I\subseteq\{1,\dots,K\}:k\in I}
  \frac{1}{\left|I\right|}
  \sum_{i\in I}
  E_i.
\end{equation}

\begin{remark}\label{rem:Shaffer}
  The inequality ``$\ge$'' in \eqref{eq:E-optimal} holds as the equality ``$=$''
  if all the intersections~\eqref{eq:H} are non-empty.
  If some of these intersections are empty, we can have a strict inequality.
  Algorithm~\ref{alg:BH} implements the definition~\eqref{eq:E*}.
  Therefore, it is valid regardless of whether some of the intersections~\eqref{eq:H} are empty;
  however, if they are, it may be possible to improve the adjusted e-values.
  According to Holm's [\citeyear{Holm:1979}] terminology, we allow ``free combinations''.
  \citet{Shaffer:1986} pioneered methods that take account of the logical relations between the base hypotheses $H_k$.
\end{remark}

To obtain Algorithm~\ref{alg:BH}, we rewrite the definitions \eqref{eq:E*}
as
\begin{align*}
  E^*_{\pi(k)}
  &=
  \min_{i\in\{0,\dots,k-1\}}
  \frac{E_{\pi(k)}+E_{(1)}+\dots+E_{(i)}}{i+1}\\
  &=
  \min_{i\in\{1,\dots,k-1\}}
  \frac{E_{\pi(k)}+E_{(1)}+\dots+E_{(i)}}{i+1}
\end{align*}
for $k\in\{1,\dots,K\}$,
where $\pi$ is the ordering permutation and $E_{(j)}=E_{\pi(j)}$ is the $j$th order statistic among $E_1,\dots,E_K$,
as in Algorithm~\ref{alg:BH}.
In lines 3--5 of Algorithm~\ref{alg:BH} we precompute the sums
\[
  S_i
  :=
  e_{(1)}+\dots+e_{(i)},
  \qquad
  i=1,\dots,K,
\]
in lines 8--9 we compute
\[
  e_{k,i}
  :=
  \frac{e_{\pi(k)}+e_{(1)}+\dots+e_{(i)}}{i+1}
\]
for $i=1,\dots,k-1$,
and as result of executing lines 6--11 we will have
\[
  e^*_{\pi(k)}
  =
  \min_{i\in\{1,\dots,k-1\}}
  e_{k,i}
  =
  \min_{i\in\{1,\dots,k-1\}}
  \frac{e_{\pi(k)}+e_{(1)}+\dots+e_{(i)}}{i+1},
\]
which shows that Algorithm~\ref{alg:BH} is an implementation of \eqref{eq:E*}.

The computational complexity of Algorithm~\ref{alg:BH} is $O(K^2)$.

\begin{algorithm}[bt]
  \caption{Adjusting sequential e-values for multiple hypothesis testing}
  \label{alg:BH-i}
  \begin{algorithmic}[1]  
    \Require
      A sequence of e-values $e_1,\dots,e_K$.
    \State Let $a$ be the product of all $e_{k}<1$, $k=1,\dots,K$ (and $a:=1$ if there are no such $k$).
    \For{$k=1,\dots,K$}
      \State $e^*_k:=a e_k$
    \EndFor
  \end{algorithmic}
\end{algorithm}

In the case of sequential  e-variables,
we have Algorithm~\ref{alg:BH-i}.
This algorithm assumes that,
under any $Q\in\mathfrak{P}(\Omega)$,
the base e-variables $E_k$, $k\in I_Q$, are sequential
(remember that $I_Q$ is defined by \eqref{eq:Q} and that independence implies being sequential).
The conditional e-variable witnessing that the output of Algorithm~\ref{alg:BH-i} is FWV
is the one given by the product ie-merging function,
\[
  E_Q
  :=
  \prod_{k\in I_Q}
  E_k,
\]
where the adjusted e-variables are defined by
\begin{equation}\label{eq:E*-i}
  E^*_k
  :=
  \min_{I\subseteq\{1,\dots,K\}:k\in I}
  \prod_{i\in I}
  E_i.
\end{equation}
A remark similar to Remark~\ref{rem:Shaffer} can also be made about Algorithm~\ref{alg:BH-i}.
The computational complexity of Algorithm~\ref{alg:BH-i} is $O(K)$
(unusually, the algorithm does not require sorting the base e-values).

\section{Merging p-values and comparisons}
\label{sec:p}

Merging p-values is a much more difficult topic than merging e-values,
but it is very well explored.
First we review merging p-values without any assumptions,
and then we move on to merging independent p-values.

A \emph{p-merging function} of $K$ p-values is an increasing Borel function $F:[0,1]^K\to[0,1]$
such that $F(P_1,\dots,P_K)\in\PPP_Q$ whenever $P_1,\dots,P_K\in\PPP_Q$.

For merging p-values without the assumption of independence,
we will concentrate on two natural families of p-merging functions.
The older family is the one introduced by \citet{Ruger:1978},
and the newer one was introduced in our paper \citet{Vovk/Wang:2019}.
R\"uger's family is parameterized by $k\in\{1,\dots,K\}$,
and its $k$th element is the function
(shown by \citet{Ruger:1978} to be a p-merging function)
\begin{equation}\label{eq:Ruger}
  (p_1,\dots,p_K)
  \mapsto
  \frac{K}{k}
  p_{(k)}
  \wedge
  1,
\end{equation}
where $p_{(k)}:=p_{\pi(k)}$ and $\pi$ is a permutation of $\{1,\dots,K\}$
ordering the p-values in the ascending order: $p_{\pi(1)}\le\dots\le p_{\pi(K)}$.
The other family \citep{Vovk/Wang:2019},
which we will refer to as the \emph{$M$-family},
is parameterized by $r\in[-\infty,\infty]$,
and its element with index $r$ has the form $a_{r,K}M_{r,K}\wedge1$,
where
\begin{equation}\label{eq:merge1}
  M_{r,K}(p_1,\dots,p_K)
  :=
  \left(
    \frac{p_1^r + \dots + p_K^r }{K}
  \right)^{1/r}
\end{equation}
and $a_{r,K}\ge1$ is a suitable constant.
We also define $M_{r,K}$ for $r\in\{0,\infty,-\infty\}$ as the limiting cases of \eqref{eq:merge1},
which correspond to the geometric average, the maximum, and the minimum, respectively. 

The initial and final elements of both families coincide:
the initial element is the Bonferroni p-merging function
\begin{equation}\label{eq:Bonferroni}
  (p_1,\dots,p_K)
  \mapsto
  K\min(p_1,\dots,p_K)
  \wedge
  1,
\end{equation}
and the final element is the maximum p-merging function
\begin{equation*} 
  (p_1,\dots,p_K)
  \mapsto
  \max(p_1,\dots,p_K).
\end{equation*}

Similarly to the case of e-merging functions,
we say that a p-merging function $F$ \emph{dominates} a p-merging function $G$
if $F\le G$.
The domination is \emph{strict} if, in addition, $F(\mathbf{p})<G(\mathbf{p})$
for at least one $\mathbf{p}\in[0,1]^K$.
We say that a p-merging function $F$ is \emph{admissible}
if it is not strictly dominated by any p-merging function~$G$.

The domination structure of p-merging functions
is much richer than that of e-merging functions.
The maximum p-merging function is clearly inadmissible
(e.g., $(p_1,\dots,p_K) \mapsto \max(p_1,\dots,p_K)$
is strictly dominated by $(p_1,\dots,p_K) \mapsto p_1$)
while the Bonferroni p-merging function is admissible,
as the following proposition shows.

\begin{proposition} 
  The Bonferroni p-merging function \eqref{eq:Bonferroni} is admissible.
\end{proposition}

\begin{proof} 
  Denote by $M_B$ the Bonferroni p-merging function \eqref{eq:Bonferroni}.
  Suppose the statement of the proposition is false
  and fix a p-merging function $F$ that strictly dominates $M_B$.
  If $F=M_B$ whenever $M_B<1$, then $F=M_B$ also when $M_B=1$, since $F$ is increasing.
  Hence for some point $(p_1,\dots,p_K)\in[0,1]^K$,
  \[
    F(p_1,\dots,p_K) < M_B(p_1,\dots,p_K) < 1.
  \]
  Fix such $(p_1,\dots,p_K)$ and set $p:=\min(p_1,\dots,p_K)$;
  we know that $K p < 1$.
  Since
  \[
    F(p,\dots,p) \le F(p_1,\dots,p_K)<M_B(p_1,\dots,p_K) = K p,
  \]
  we can take $\epsilon\in(0,p)$ such that $F(p,\dots,p)< K(p-\epsilon)$.
  Let $A_1,\dots,A_K,B$ be disjoint events such that $Q(A_k)=p-\epsilon$ for all $k$ and $Q(B)=\epsilon$
  (their existence is guaranteed by the inequality $K p < 1$).
  Define random variables
  \[
    U_k
    :=
    \begin{cases}
      p-\epsilon & \text{if $A_k$ happens}\\
      p & \text{if $B$ happens}\\
      1 & \text{otherwise},
    \end{cases}
  \]
  $k=1,\dots,K$.
  It is straightforward to check that $U_1,\dots,U_K \in\PPP_Q$.
  By writing $F:=F(U_1,\dots,U_K)$ and $M_B:=M_B(U_1,\dots,U_K)$, 
  we have
  \begin{align*}
    Q(F \le K(p-\epsilon))
    &=
    Q(M_B \le K(p-\epsilon)) + Q(F \le K(p-\epsilon)< M_B)\\
    &\ge
    Q(\min(U_1,\dots,U_K) \le p-\epsilon) + Q(U_1=\dots=U_k=p)\\
    &=
    Q
    \left(
      \bigcup_{k=1}^K A_k
    \right)
    +
    Q(B)
    =
    \sum_{k=1}^K Q(A_k) + \epsilon\\
    &=
    K(p-\epsilon)+\epsilon
    >
    K(p-\epsilon).
  \end{align*}
  Therefore, $F$ is not a p-merging function,
  which gives us the desired contradiction.
\end{proof}

The general domination structure of p-merging functions appears to be very complicated,
and is the subject of \citet{Vovk/etal:2020}.

\subsection*{Connections to e-merging functions}

The domination structure of the class of e-merging functions is very simple,
according to Theorem~\ref{thm:iff}.
It makes it very easy to understand what the e-merging analogues
of R\"uger's family and the $M$-family are;
when stating the analogues we will use the rough relation $1/e\sim p$
between e-values and p-values (see Remark \ref{rem:rough}).
Let us say that an e-merging function $F$ is \emph{precise} if $c F$ is not an e-merging function for any $c>1$.

For a sequence $e_1,\dots,e_K$,
let $e_{[k]}:=e_{\pi(k)}$ be the order statistics numbered from  the largest to the smallest;
here $\pi$ is a permutation of $\{1,\dots,K\}$
ordering $e_k$ in the descending order: $e_{\pi(1)}\ge\dots\ge e_{\pi(K)}$.
Let us check that the R\"uger-type function $(e_1,\dots,e_K)\mapsto(k/K)e_{[k]}$ is a precise e-merging function.
It is an e-merging function since it is dominated by the arithmetic mean:
indeed, the condition of domination
\begin{equation}\label{eq:domination}
  \frac{k}{K} e_{[k]}
  \le
  \frac{e_1+\dots+e_K}{K},
\end{equation}
can be rewritten as
\[
  k e_{[k]}
  \le
  e_1+\dots+e_K
\]
and so is obvious.
As sometimes we have a strict inequality, the e-merging function is inadmissible
(remember that we assume $K\ge2$).
The e-merging function is precise because \eqref{eq:domination} holds as equality
when the $k$ largest $e_i$, $i\in\{1,\dots,K\}$, are all equal and greater than 1
and all the other $e_i$ are 0.

In the case of the $M$-family,
let us check that the function
\begin{equation}\label{eq:M}
  F
  :=
  (K^{1/r-1}\wedge1) M_{r,K}
\end{equation}
is a precise e-merging function,
for any $r\in[-\infty,\infty]$.
For $r\le1$, $M_{r,K}$ is increasing in $r$
\citep[Theorem~16]{Hardy/etal:1952},
and so $F = M_{r,K}$ is dominated by the arithmetic mean $M_K$;
therefore, it is an e-merging function.
For $r>1$ we can rewrite the function $F=K^{1/r-1}M_{r,K}$ as
\begin{equation*}
  F(e_1,\dots,e_K)
  =
  K^{1/r-1} M_{r,K}(e_1,\dots,e_K)
  =
  K^{-1}
  \left(
    e_1^r + \dots + e_K^r
  \right)^{1/r},
\end{equation*}
and we know that the last expression is a decreasing function of $r$
\citep[Theorem~19]{Hardy/etal:1952};
therefore, $F$ is also dominated by $M_K$ and so is a merging function.
The e-merging function $F$ is precise (for any $r$) since
\begin{align*}
  r\le 1
  &\Longrightarrow
  F(e,\dots,e) = M_K(e,\dots,e)=e\\
  r>1
  &\Longrightarrow
  F(0,\dots,0,e) = M_K(0,\dots,0,e)=e/K,
\end{align*}
and so by Proposition~\ref{prop:M} (applied to a sufficiently large $e$)
$c F$ is not an e-merging function for any $c>1$.
But $F$ is admissible if and only if $r=1$
as shown by Theorem~\ref{thm:iff}.

\begin{remark}
  The rough relation $1/e \sim p$ also sheds light on the coefficient,
  $K^{1/r-1}\wedge1 = K^{1/r-1}$ for $r>1$,
  given in \eqref{eq:M} in front of $M_{r,K}$.
  The coefficient $K^{1/r-1}$, $r>1$, in front of $M_{r,K}$ for averaging e-values
  corresponds to a coefficient of $K^{1+1/r}$, $r<-1$, in front of $M_{r,K}$ for averaging p-values.
  And indeed, by Proposition~5 of \citet{Vovk/Wang:2019},
  the asymptotically precise coefficient in front of $M_{r,K}$, $r<-1$, for averaging p-values is $\frac{r}{r+1}K^{1+1/r}$.
  The extra factor $\frac{r}{r+1}$ appears
  because the reciprocal of a p-variable is only approximately, but not exactly, an e-variable.
\end{remark}
 
\begin{remark}
  Our formulas for merging e-values are explicit
  and much simpler than the formulas for merging p-values given in \citet{Vovk/Wang:2019}, where the coefficient $a_{r,K}$ is often not analytically available. 
  Merging e-values does not involve asymptotic approximations via the theory of robust risk aggregation
  (e.g., \citet{Embrechts/etal:2015}),
  as used in that paper.
  This suggests that in some important respects e-values are easier objects to deal with than p-values.
\end{remark}

\subsection*{Merging independent p-values}

In this section we will discuss ways of combining p-values $p_1,\dots,p_K$
under the assumption that the p-values are independent.

One of the oldest and most popular methods for combining p-values is Fish\-er's [\citeyear{Fisher:1932}, Section~21.1],
which we already mentioned in Section~\ref{sec:ie}.
Fisher's method is based on the product statistic $p_1\dots p_K$
(with its low values significant)
and uses the fact that $-2\ln(p_1\dots p_K)$ has the $\chi^2$ distribution with $2K$ degrees of freedom
when $p_k$ are all independent and distributed uniformly on the interval $[0,1]$;
the p-values are the tails of the $\chi^2$ distribution.

\citet{Simes:1986} proves a remarkable result for R\"uger's family \eqref{eq:Ruger}
under the assumption that the p-values are independent:
the minimum
\begin{equation}\label{eq:Simes}
  (p_1,\dots,p_K)
  \mapsto
  \min_{k\in\{1,\dots,K\}}
  \frac{K}{k}
  p_{(k)}
\end{equation}
of R\"uger's family over all $k$ turns out to be a p-merging function.
The counterpart of Simes's result still holds for e-merging functions;
moreover, now the input e-values do not have to be independent.
Namely,
\begin{equation*} 
  (e_1,\dots,e_K)
  \mapsto
  \max_{k\in\{1,\dots,K\}}
  \frac{k}{K}
  e_{[k]}
\end{equation*}
is an e-merging function.
This follows immediately from \eqref{eq:domination},
the left-hand side of which can be replaced by its maximum over $k$.
And it also follows from \eqref{eq:domination}
that there is no sense in using this counterpart;
it is better to use the arithmetic mean.

\section{Experimental results}
\label{sec:experiments}

In this section we will explore the performance of various methods of combining e-values and p-values
and multiple hypothesis testing,
both standard and introduced in this paper.
For our code, see \citet{Vovk/Wang:code}.

In order to be able to judge how significant results of testing using e-values are,
Jeffreys's [\citeyear[Appendix B]{Jeffreys:1961}] rule of thumb may be useful:
\begin{itemize}
\item
  If the resulting e-value $e$ is below 1,
  the null hypothesis is supported.
\item
  If $e\in(1,\sqrt{10})\approx(1,3.16)$,
  the evidence against the null hypothesis is not worth more than a bare mention.
\item
  If $e\in(\sqrt{10},10)\approx(3.16,10)$,
  the evidence against the null hypothesis is substantial.
\item
  If $e\in(10,10^{3/2})\approx(10,31.6)$,
  the evidence against the null hypothesis is strong.
\item
  If $e\in(10^{3/2},100)\approx(31.6,100)$,
  the evidence against the null hypothesis is very strong.
\item
  If $e>100$,
  the evidence against the null hypothesis is decisive.
\end{itemize}

Our discussions in this section assume that our main interest is in e-values,
and p-values are just a possible tool for obtaining good e-values
(which is, e.g., the case for Bayesian statisticians in their attitude towards Bayes factors and p-values;
cf.\ Section~\ref{subsec:Bayes-factors} and Appendix~\ref{app:calibrators}).
Our conclusions would have been different had our goal been to obtain good p-values.

\subsection*{Combining independent e-values and p-values}

\begin{figure}
  \begin{center}
    \includegraphics[width=0.6\textwidth]{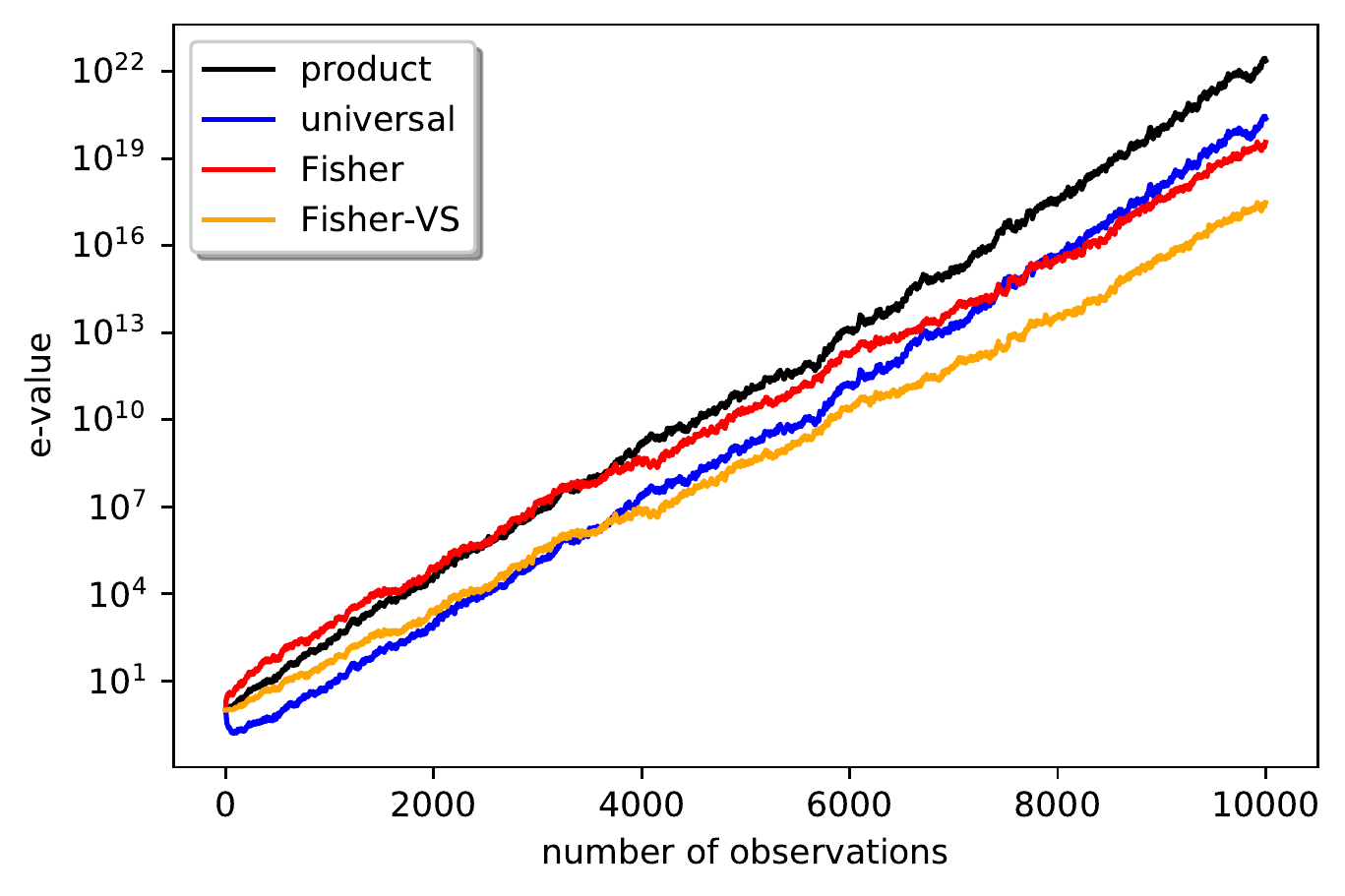}
  \end{center}
  \caption{Combining p-values using Fisher's method vs combining e-values by multiplication
    (details in text).}
  \label{fig:combining}
\end{figure}

First we explore combining independent e-values and independent p-values; see Figure~\ref{fig:combining}.
The observations are generated from the Gaussian model $N(\mu,1)$ with standard deviation 1 and unknown mean $\mu$.
The null hypothesis is $\mu=0$ and the alternative hypothesis is $\mu=\delta$;
for Figures~\ref{fig:combining} and~\ref{fig:combining_ABS} we set $\delta:=-0.1$.
The observations are IID.
Therefore, one observation does not carry much information about which hypothesis is true,
but repeated observations quickly reveal the truth (with a high probability).

For Figures~\ref{fig:combining} and \ref{fig:combining_ABS},
all data ($10{,}000$ or $1000$ observations, respectively) are generated from the alternative distribution
(there will be an example where some of the data is coming from the null distribution
in Appendix~\ref{app:extra-experiments}).
For each observation, the e-value used for testing is the likelihood ratio
\begin{equation}\label{eq:data-1}
  E(x)
  :=
  e^{-(x-\delta)^2/2} / e^{-x^2/2}
  =
  e^{x\delta-\delta^2/2}
\end{equation}
of the alternative probability density to the null probability density,
where $x$ is the observation.
It is clear that \eqref{eq:data-1} is indeed an e-variable under the null hypothesis:
its expected value is 1.
As the p-value we take
\begin{equation}\label{eq:p-1}
  P(x)
  :=
  N(x),
\end{equation}
where $N$ is the standard Gaussian distribution function;
in other words, the p-value is found using the most powerful test,
namely the likelihood ratio test given by the Neyman--Pearson lemma.

In Figure~\ref{fig:combining} we give the results for the product e-merging function \eqref{eq:product}
and Fisher's method described in the last subsection of Section~\ref{sec:p}.
(The other methods that we consider are vastly less efficient,
and we show them in the following figure,
Figure~\ref{fig:combining_ABS}.)
Three of the values plotted in Figure~\ref{fig:combining} against each $K=1,\dots,10{,}000$ are:
\begin{itemize}
\item
  the product e-value $E(x_1)\dots E(x_K)$;
  it is shown as the black line;
\item
  the reciprocal $1/p$ of Fisher's p-value $p$ obtained by merging the first $K$ p-values $P(x_1),\dots,P(x_K)$;
  it is shown as the red line;
\item
  the VS bound applied to Fisher's p-value;
  it is shown as the orange line.
\end{itemize}
The plot depends very much on the seed for the random number generator,
and so we report the median of all values over 100 seeds.

The line for the product method is below that for Fisher's over the first 2000 observations
but then it catches up.
If our goal is to have an overall e-value summarizing the results of testing based on the first $K$ observations
(as we always assume in this section),
the comparison is unfair, since Fisher's p-values need to be calibrated.
A fairer (albeit still unfair) comparison is with the VS bound,
and the curve for the product method can be seen to be above the curve for the VS bound.
\emph{A fortiori},
the curve for the product method would be above the curve for any of the calibrators in the family~\eqref{eq:calibrator}.

It is important to emphasize that the natures of plots for e-values and p-values are very different.
For the red and orange lines in Figure~\ref{fig:combining},
the values shown for different $K$
relate to different batches of data
and cannot be regarded as a trajectory of a natural stochastic process.
In contrast, the values shown by the black line for different $K$ are updated sequentially,
the value at $K$ being equal to the value at $K-1$ multiplied by $E(x_K)$,
and form a trajectory of
a test martingale.
Moreover, for the black line we do not need the full force of the assumption of independence of the p-values.
As we discuss at the end of Section~\ref{sec:ie},
it is sufficient to assume that $E(x_K)$ is a valid e-value given $x_1,\dots,x_{K-1}$;
the black line in Figure~\ref{fig:combining} is then still a trajectory of a test supermartingale.

What we said in the previous paragraph
can be regarded as an advantage of using e-values.
On the negative side,
computing good (or even optimal in some sense) e-values often requires more detailed knowledge.
For example, whereas computing the e-value \eqref{eq:data-1} requires the knowledge of the alternative hypothesis,
for computing the p-value \eqref{eq:p-1} it is sufficient to know that the alternative hypothesis
corresponds to $\mu<0$.
Getting $\mu$ very wrong will hurt the performance of methods based on e-values.
To get rid of the dependence on
$\mu$,
we can, e.g., integrate the product e-value over $\delta\sim N(0,1)$
(taking the standard deviation of 1 is somewhat wasteful in this situation,
but we take the most standard probability measure).
This gives the ``universal'' test martingale (see e.g., \citet{Howard/etal:arXiv1810})
\begin{align}
  S_K
  &:=
  \frac{1}{\sqrt{2\pi}}
  \int_{-\infty}^{\infty}
  \exp(-\delta^2/2)
  \prod_{k=1}^K
  \exp(x_k\delta-\delta^2/2)
  \d\delta\notag\\
  &=
  \frac{1}{\sqrt{K+1}}
  \exp
  \left(
    \frac{1}{2(K+1)}
    \left(
      \sum_{k=1}^K
      x_k
    \right)^2
  \right).
  \label{eq:universal}
\end{align}
This test supermartingale is shown in blue in Figure~\ref{fig:combining}.
It is below the black line but at the end of the period it catches up even with the line for Fisher's method
(and beyond that period it overtakes Fisher's method more and more convincingly).

\begin{figure}
  \begin{center}
    \includegraphics[width=0.6\textwidth]{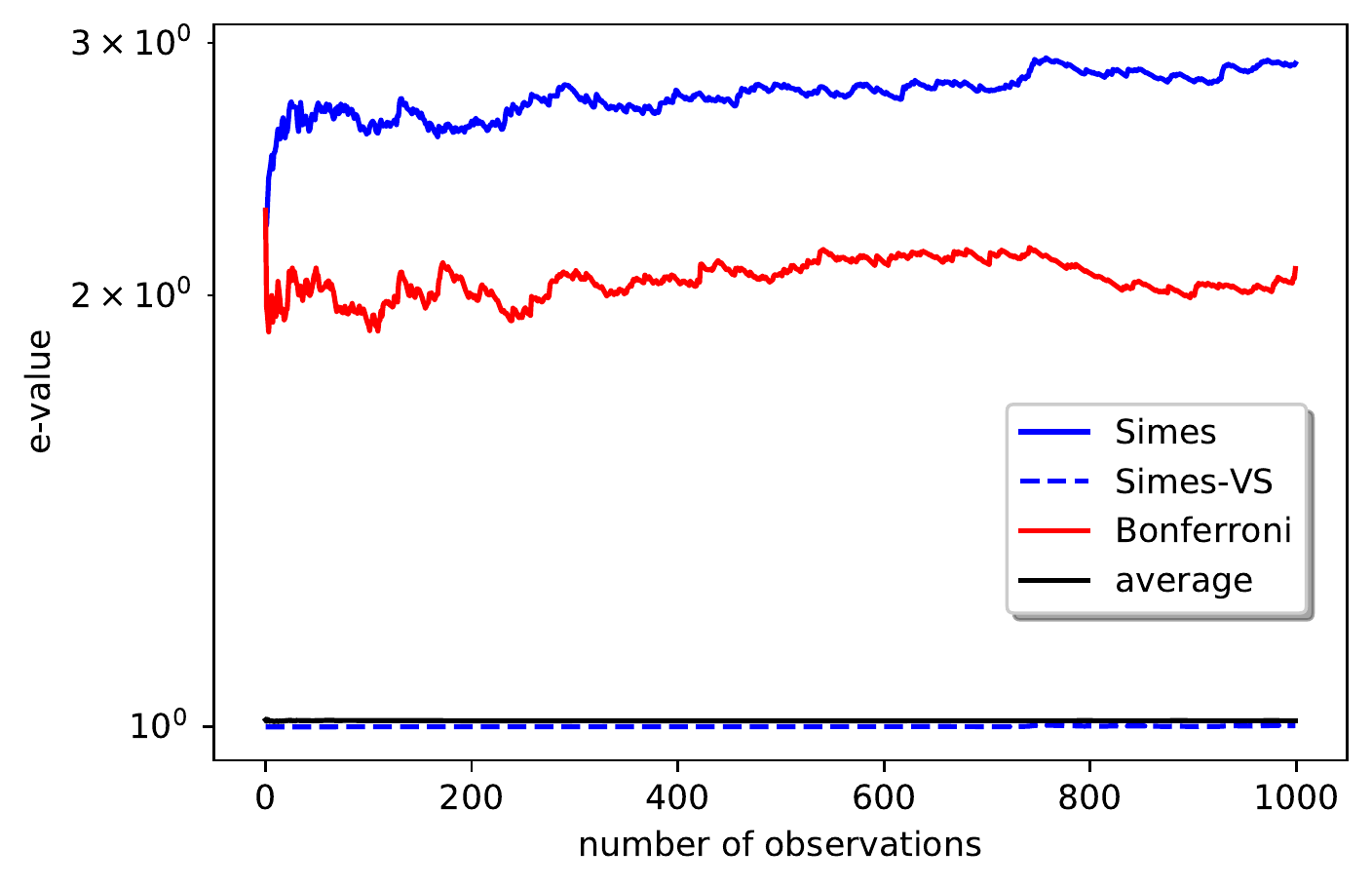}
  \end{center}
  \caption{Combining p-values using Simes's and Bonferroni's methods
    and combining e-values using averaging (details in text).}
  \label{fig:combining_ABS}
\end{figure}

Arithmetic average \eqref{eq:average} and Simes's method \eqref{eq:Simes}
have very little power in the situation of Figure~\ref{fig:combining}:
see Figure~\ref{fig:combining_ABS},
which plots the e-values produced by the averaging method,
the reciprocals $1/p$ of Simes's p-values $p$,
the VS bound for Simes's p-values,
and the reciprocals of the Bonferroni p-values
over 1000 observations,
all averaged (in the sense of median) over 1000 seeds.
They are very far from attaining statistical significance (a p-value of $5\%$ or less)
or collecting substantial evidence against the null hypothesis
(an e-value of $\sqrt{10}$ or more according to Jeffreys).

\subsection*{Multiple hypothesis testing}

\begin{figure}
  \begin{center}
    \includegraphics[width=0.8\textwidth]{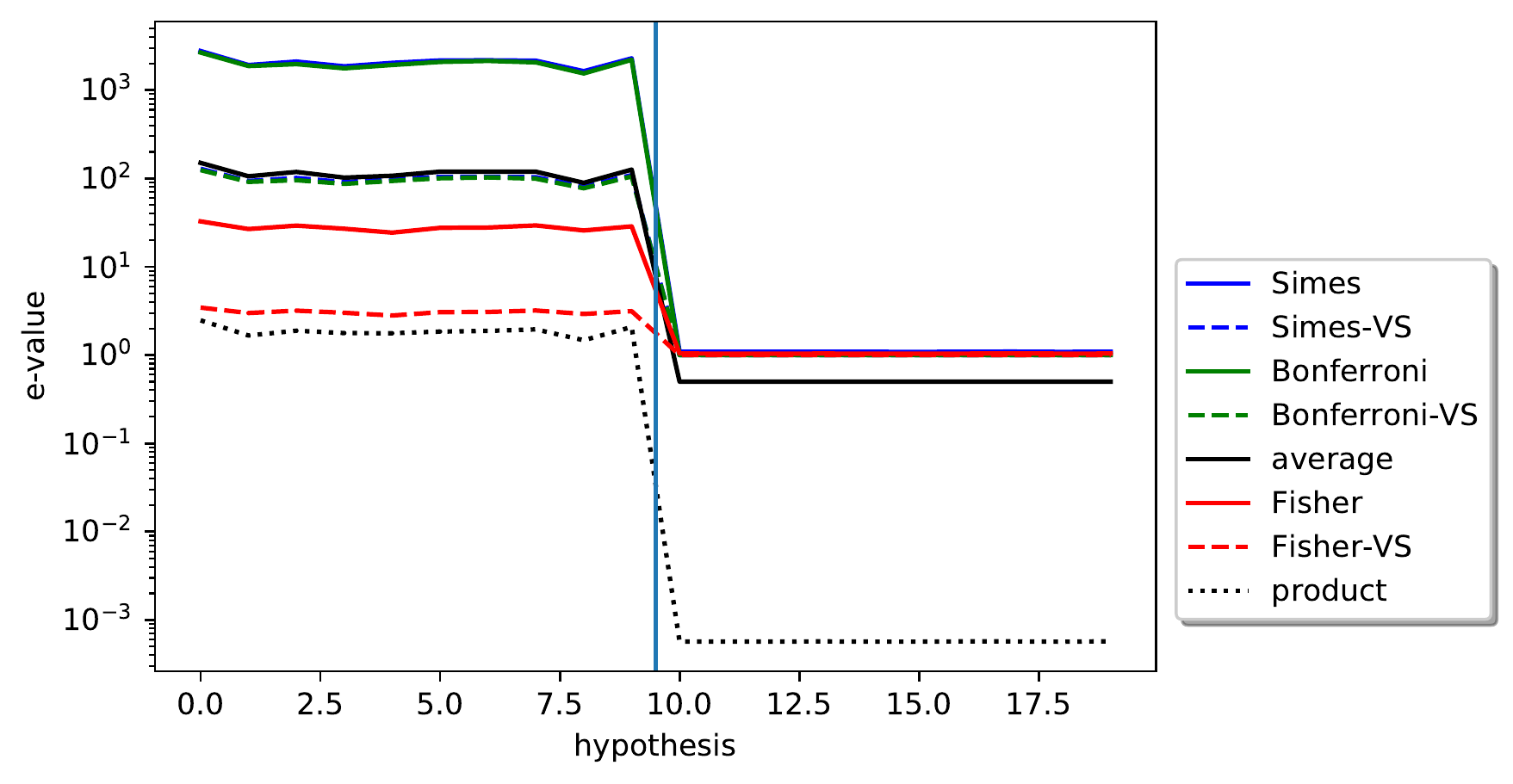}
  \end{center}
  \caption{Multiple hypothesis testing for 20 hypotheses using p-values and e-values,
    with some graphs indistinguishable (details in text).}
  \label{fig:multiple_small}
\end{figure}

Next we discuss multiple hypothesis testing.
Figure~\ref{fig:multiple_small} shows plots of adjusted e-values and adjusted p-values
resulting from various methods for small numbers of hypotheses,
including Algorithms~\ref{alg:BH} and~\ref{alg:BH-i}.
The observations are again generated from the statistical model $N(\mu,1)$.

We are testing 20 null hypotheses.
All of them are $\mu=0$, and their alternatives are $\mu=-4$.
Each null hypothesis is tested given an observation drawn either from the null or from the alternative.
The first 10 null hypotheses are false, and in fact the corresponding observations
are drawn from the alternative distribution.
The remaining 10 null hypotheses are true,
and the corresponding observations are drawn from them rather than the alternatives.
The vertical blue line at the centre of Figure~\ref{fig:multiple_small}
separates the false null hypotheses from the true ones:
null hypotheses 0 to 9 are false and 10 to 19 are true.
We can see that at least some of the methods can detect that the first 10 null hypotheses are false.


Since some of the lines are difficult to tell apart,
we will describe the plot in words.
The top two horizontal lines to the left of the vertical blue line are indistinguishable
but are those labeled as Simes and Bonferroni in the legend;
they correspond to e-values around $2\times10^3$.
The following cluster of horizontal lines to the left of the vertical blue line
(with e-values around $10^2$) are those labeled as average, Simes-VS, and Bonferroni-VS,
with average slightly higher.
To the right of the vertical blue line,
the upper horizontal lines (with e-values $10^0$) include all methods except for average and product;
the last two are visible.

Most of the methods (all except for Bonferroni and Algorithm~\ref{alg:BH})
require the observations to be independent.
The base p-values are \eqref{eq:p-1},
and the base e-values are the likelihood ratios
\begin{equation}\label{eq:data-2}
  E(x)
  :=
  \frac12
  e^{x\delta-\delta^2/2}
  +
  \frac12
\end{equation}
(cf.\ \eqref{eq:data-1}) of the ``true'' probability density to the null probability density,
where the former assumes that the null or alternative distribution for each observation
is decided by coin tossing.
Therefore, the knowledge encoded in the ``true'' distribution is that
half of the observations are generated from the alternative distribution,
but it is not known that these observations are in the first half.
We set $\delta:=-4$ in \eqref{eq:data-2},
keeping in mind that accurate prior knowledge is essential for the efficiency of methods based on e-values. 


A standard way of producing multiple testing procedures
is applying the closure principle described in Appendix~\ref{app:survey}
and already implicitly applied in Section~\ref{sec:multiple}
to methods of merging e-values.
In Figure~\ref{fig:multiple_small} we report the results for the closures of five methods,
three of them producing p-values (Simes's, Bonferroni's, and Fisher's)
and two producing e-values (average and product);
see Section~\ref{sec:multiple} for self-contained descriptions
of the last two methods (Algorithms~\ref{alg:BH} and~\ref{alg:BH-i}).
For the methods producing p-values we show the reciprocals $1/p$ of the resulting p-values $p$
(as solid lines)
and the corresponding VS bounds (as dashed lines).
For the closure of Simes's method we follow the appendix of \citet{Wright:1992},
the closure of Bonferroni's method is described in \citet{Holm:1979}
(albeit not in terms of adjusted p-values),
and for the closure of Fisher's method we use
Dobriban's [\citeyear{Dobriban:2020}] FACT (FAst Closed Testing) procedure.
To make the plot more regular, all values are averaged
(in the sense of median)
over 1000 seeds of the Numpy random number generator.

According to Figure~\ref{fig:multiple_small},
the performance of Simes's and Bonferroni's methods is very similar,
despite Bonferroni's method not depending on the assumption of independence of the p-values.
The e-merging method of averaging (i.e., Algorithm~\ref{alg:BH})
produces better e-values than those obtained by calibrating the closures of Simes's and Bonferroni's methods;
remember that the line corresponding to Algorithm~\ref{alg:BH} should be compared
with the VS versions (blue and green dashed, which almost coincide) of the lines
corresponding to the closures of Simes's and Bonferroni's methods,
and even that comparison is unfair and works in favour of those two methods
(since the VS bound is
not a valid calibrator).
The other algorithms perform poorly.

\begin{figure}
  \begin{center}
    \includegraphics[width=0.6\textwidth]{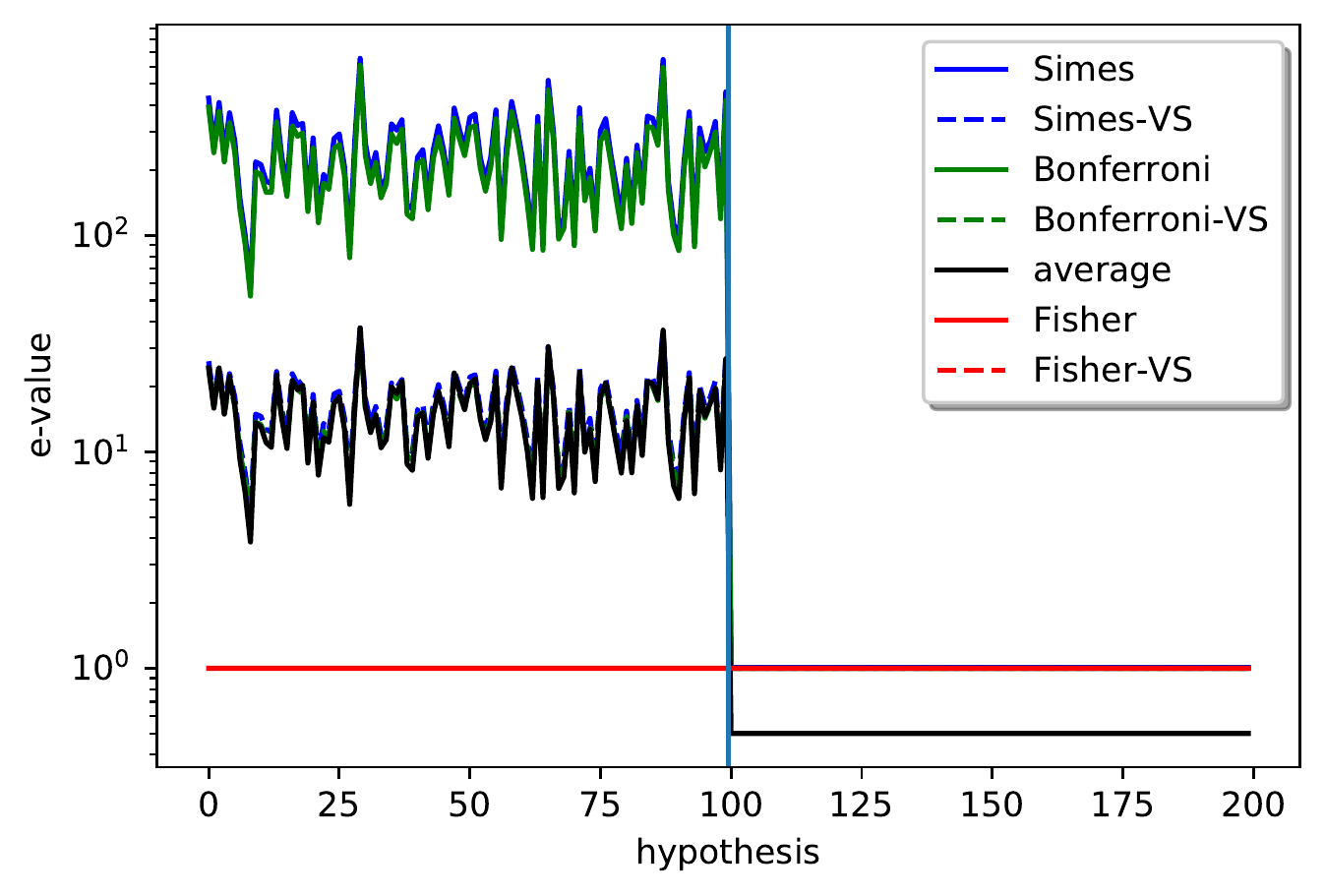}
  \end{center}
  \caption{The analogue of Figure~\ref{fig:multiple_small} without the product method,
    with 200 observations, and with some graphs indistinguishable (details in text).}
  \label{fig:multiple_big}
\end{figure}

Figure~\ref{fig:multiple_big} is an analogue of Figure~\ref{fig:multiple_small}
that does not show results for merging by multiplication
(for large numbers of hypotheses its results are so poor that, when shown,
differences between the other methods become difficult to see).
To get more regular and comparable graphs,
we use averaging (in the sense of median) over 100 seeds.

Since some of the graphs coincide, or almost coincide,
we will again describe the plot in words
(referring to graphs that are straight or almost straight as lines).
To the left of the vertical blue line
(separating the false null hypotheses 0--99 from the true null hypotheses 100--199)
we have three groups of graphs:
the top graphs (with e-values around $2\times10^2$) are those labeled as Simes and Bonferroni in the legend,
the middle graphs (with e-values around $10^1$) are those labeled as average, Simes-VS, and Bonferroni-VS,
and the bottom lines (with e-values around $10^0$) are those labeled as Fisher and Fisher-VS.
To the right of the vertical blue line,
we have two groups of lines:
the upper lines (with e-values $10^0$) include all methods except for average,
which is visible.

Now the graph for the averaging method (Algorithm~\ref{alg:BH})
is very close to (barely distinguishable from)
the graphs for the VS versions of the closures of Simes's and Bonferroni's methods,
which is a very good result (in terms of the quality of e-values that we achieve):
the VS bound is a bound on what can be achieved whereas the averaging method produces a bona fide e-value.
  The two lines (solid and dotted) for Fisher's method are indistinguishable from the horizontal axis;
  the method does not scale up in our experiments
  (which is a known phenomenon in the context of p-values:
  see, e.g., \citet[Section 1]{Westfall:2011}).
  And the four blue and green lines (solid and dotted) for Simes's and Bonferroni's methods are not visible to the right of 100
  since they are covered by the lines for Fisher's method.
  The behaviour of the lines for Simes's, Bonferroni's, and Fisher's methods to the right of 100 demonstrates
  that they do not produce valid e-values:
  for validity,
  we have to pay by getting e-values below 1 when the null hypothesis is true
  in order to be able to get large e-values when the null hypothesis is false
  (which is the case for the averaging method, represented by the black line).
  Most of these remarks are also applicable to  Figure~\ref{fig:multiple_small}.

A key advantage of the averaging and Bonferroni's methods over Simes's and Fisher's
is that they are valid regardless of whether the base e-values or p-values are independent.

\section{Conclusion}
\label{sec:conclusion}

This paper systematically explores the notion of an e-value,
which can be regarded as a betting counterpart of p-values
that is much more closely related to Bayes factors and likelihood ratios.
We argue that e-values often are more mathematically convenient than p-values
and lead to simpler results.
In particular, they are easier to combine:
the average of e-values is an e-value, and the product of independent e-values is an e-value.
We apply e-values in two areas,
multiple testing of a single hypothesis and testing multiple hypotheses,
and obtain promising experimental results.
One of our experimental findings is that,
for testing multiple hypotheses,
the performance of the most natural method based on e-values
almost attains the Vovk--Sellke bound for the closure of Simes's method,
despite that bound being overoptimistic and not producing bona fide e-values.

\subsection*{Acknowledgments}

The authors thank Aaditya Ramdas, Alexander Schied, and Glenn Shafer for helpful suggestions.
Thoughtful comments by the associate editor and four reviewers of the journal version
have led to numerous improvements in presentation and substance.

V.~Vovk's research has been partially supported by Amazon, Astra Zeneca, and Stena Line.
R.~Wang is supported by  the Natural Sciences and Engineering Research Council of Canada 
(RGPIN-2018-03823, RGPAS-2018-522590).

\appendix
\numberwithin{equation}{section}

\section{Comparisons with existing literature}
\label{app:survey}

\subsection{Bayes factors}
\label{subsec:Bayes-factors}

Historically, the use of p-values versus e-values reflects the conventional division of statistics
into frequentist and Bayesian
(although a sizable fraction of people interested in the foundations of statistics,
including the authors of this paper,
are neither frequentists nor Bayesians).
P-values are a hallmark of frequentist statistics,
but Bayesians often regard p-values as misleading,
preferring the use of Bayes factors
(which can be combined with prior probabilities to obtain posterior probabilities).
In the case of simple statistical hypotheses,
a Bayes factor is the likelihood ratio of an alternative hypothesis to the null hypothesis
(or vice versa, as in \citet{Shafer/etal:2011}).
From the betting point of view of this paper,
the key property of the Bayes factor is that it is an e-variable.

For composite hypotheses, Bayes factors and e-values diverge.
For example, a possible general definition of a Bayes factor is as follows
\citep[Section~1.2]{Kamary/etal:arXiv1412}.
Let $(f^0_{\theta}\mid\theta\in\Theta_0)$ and $(f^1_{\theta}\mid\theta\in\Theta_1)$
be two statistical models on the same sample space $\Omega$,
which is a measurable space, $(\Omega,\AAA)$, with a fixed measure $P$,
and $\Theta_1$ and $\Theta_2$ are measurable spaces.
Each $f^n_{\theta}(\omega)$, $n\in\{0,1\}$, is a probability density as function of $\omega$
and a measurable function of $\theta\in\Theta_n$.
The corresponding families of probability measures are
$(f^0_{\theta}P)_{\theta\in\Theta_0}$ and $(f^1_{\theta}P)_{\theta\in\Theta_1}$,
where $f P$ is defined as the probability measure $(f P)(A):=\int_A f\d P$, $A\in\AAA$.
Make them Bayesian models by fixing prior probability distributions $\mu_0$ and $\mu_1$
on $\Theta_0$ and $\Theta_1$, respectively.
This way we obtain Bayesian analogues of the null and alternative hypotheses, respectively.
The corresponding \emph{Bayes factor} is
\begin{equation}\label{eq:B}
  B(\omega)
  :=
  \frac
    {\int_{\Theta_1}f^1_{\theta}(\omega)\mu_1(\dd\theta)}
    {\int_{\Theta_0}f^0_{\theta}(\omega)\mu_0(\dd\theta)},
  \quad
  \omega\in\Omega.
\end{equation}
If $\Theta_0$ is a singleton,
then $B$ is an e-variable for the probability measure $Q:=f^0 P$.
In general, however, this is no longer true.
Remember that, according to our definition in Section~\ref{sec:multiple},
for $B$ to be an e-variable w.r.\ to the null hypothesis $\Theta_0$
it needs to satisfy $\int B f^0_{\theta} \d P\le1$ for all $\theta\in\Theta_0$.
However, \eqref{eq:B} only guarantees this property ``on average'',
$\int B f^0_{\theta} \d P\mu_0(\dd\theta)\le1$.
Therefore, for a composite null hypothesis a Bayes factor does not need to be an e-value
w.r.\ to that null hypothesis
(it is an e-value w.r.\ to its average).

The literature on Bayes factors is vast;
we only mention the fundamental book by \citet{Jeffreys:1961},
the influential review by \citet{Kass/Raftery:1995},
and the historical investigation by \citet{Etz/Wagenmakers:2017}.
Jeffreys's scale that we used in Section~\ref{sec:experiments}
was introduced in the context of Bayes factors,
but of course it is also applicable to e-values
in view of the significant overlap between the two notions.
\citet[Section 3.2]{Kass/Raftery:1995} simplify Jeffreys's scale
by merging the ``strong'' and ``very strong'' categories into one,
which they call ``strong''.

\subsection{Algorithmic theory of randomness}
\label{subsec:ATR}

One area where both p-values and e-values have been used for a long time
is the algorithmic theory of randomness
(see, e.g., \citet{Shen/etal:2017}),
which originated in Kolmogorov's work on the algorithmic foundations of probability and information
\citep{Kolmogorov:1965,Kolmogorov:1968}.
\citet{Martin-Lof:1966} introduced an algorithmic version of p-values,
and then \citet{Levin:1976} introduced an algorithmic version of e-values.
In the algorithmic theory of randomness
people are often interested in low-accuracy results,
and then p-values and e-values can be regarded as slight variations of each other:
if $e$ is an e-value, $1/e$ will be a p-value;
and vice versa, if $p$ is a p-value, $1/p$ will be an approximate e-value.
We discussed this approximation in detail in the main paper;
see, e.g., Remark~\ref{rem:rough}.

\subsection{Standard methods of multiple hypothesis testing}
\label{subsec:MHT}

Let us check what the notion of family-wise validity becomes
when p-variables are used instead of e-variables.
Now we have a procedure that,
given p-variables $P_k$ for testing $H_k$, $k\in\{1,\dots,K\}$,
produces random variables $P^*_1,\dots,P^*_K$ taking values in $[0,1]$.
A \emph{conditional p-variable} is a family of p-variables $P_Q$,
$Q\in\mathfrak{P}(\Omega)$.
The procedure's output $P^*_1,\dots,P^*_K$ is \emph{family-wise valid} (\emph{FWV})
if there exists a conditional p-variable $(P_Q)_{Q\in\mathfrak{P}(\Omega)}$ such that
\begin{equation}\label{eq:p-FWV}
  \forall k\in\{1,\dots,K\} \;
  \forall Q\in H_k:
  P_Q \le P^*_k.
\end{equation}
In this case we can see that, for any $Q\in\mathfrak{P}(\Omega)$ and any $\epsilon\in(0,1)$,
\begin{equation}\label{eq:FWER}
  Q(\exists k\in\{1,\dots,K\}: Q\in H_k \text{ and } P^*_k\le\epsilon)
  \le
  Q(P_Q\le\epsilon)
  \le
  \epsilon.
\end{equation}
The left-most expression in \eqref{eq:FWER} is known as the \emph{family-wise error rate}
(the standard abbreviation is FWER)
of the procedure that rejects $H_k$ when $P^*_k\le\epsilon$.
The inequality between the extreme terms of \eqref{eq:FWER} can be expressed
as $P^*_k$ being \emph{family-wise adjusted p-values}.
(See, e.g., \citet[Section 3.2]{Efron:2010}.)

On the other hand, we can check that any procedure satisfying \eqref{eq:FWER}
will satisfy \eqref{eq:p-FWV} for some conditional p-variable $(P_Q)$:
indeed, we can set
\[
  P_Q
  :=
  \min_{Q\in H_k}
  P^*_k.
\]

\begin{remark} 
  Notice that calibrators maintain the FWV property.
  Namely, if p-variables $P^*_1,\dots,P^*_K$ are FWV
  and $f$ is a calibrator,
  the e-variables $f(P^*_1),\dots,f(P^*_K)$ are FWV.
  This follows immediately from the definitions~\eqref{eq:FWV} and~\eqref{eq:p-FWV}.
  And in the opposite direction, if e-variables $E^*_1,\dots,E^*_K$ are FWV
  and $g$ is an e-to-p-calibrator,
  the p-variables $g(E^*_1)\wedge1,\dots,g(E^*_K)\wedge1$ are FWV.
\end{remark}

As we mentioned in Section~\ref{sec:multiple},
Algorithms~\ref{alg:BH} and~\ref{alg:BH-i} can be obtained from the e-merging function \eqref{eq:average}
by applying the closure principle.
In our description of this principle we will follow \citet[Section 3.3]{Efron:2010}.
Suppose, for some $\epsilon>0$ and all $I\subseteq\{1,\dots,K\}$,
we have a level-$\epsilon$ test function $\phi_I:\Omega\to\{0,1\}$:
\[
  \forall Q\in\cap_{i\in I}H_i:
  \E^Q[\phi_I]
  \le
  \epsilon;
\]
$\phi_I=1$ means that the combined null hypothesis $\cap_{i\in I}H_i$ is rejected.
(Such a collection of ``local tests'', for all $I$ and $\epsilon$,
is just a different representation of p-merging functions.)
The principle then recommends the simultaneous test function
\[
  \Phi_J
  :=
  \min_{I\supseteq J}
  \phi_I,
  \qquad
  J\subseteq\{1,\dots,K\};
\]
this simultaneous test function rejects $J$
if $\phi$ rejects all $I$ such that $J\subseteq I\subseteq\{1,\dots,K\}$.
If $P_1,\dots,P_K$ are p-variables, $f$ is a symmetric p-merging function, and $\phi$ is defined by
\[
  \phi_I = 1
  \Longleftrightarrow
  f(P_i,i\in I) \le \epsilon
\]
(which is clearly a level-$\epsilon$ test function),
we have
\[
  \Phi_J = 1
  \Longleftrightarrow
  \max_{I\supseteq J}
  f(P_i,i\in I)
  \le
  \epsilon
\]
(omitting the dependence of $\phi$ and $\Phi$ on $\epsilon$).
This corresponds to the simultaneous p-variable
\begin{equation}\label{eq:J}
  P_J
  :=
  \max_{I\supseteq J}
  f(P_i,i\in I).
\end{equation}
In this paper we are only interested in the case where $J$ is a singleton
(analogues for general $J$ are considered in \citet{Vovk/Wang:arXiv1912,Vovk/Wang:arXiv2003},
to be discussed later).
This gives us the adjusted p-values
\[
  P^*_k = P_{\{k\}}
  :=
  \max_{I\ni k}
  f(P_i,i\in I).
\]
The corresponding formula for the adjusted e-values is
\[
  E^*_k
  :=
  \min_{I\ni k}
  f(E_i,i\in I).
\]
This coincides with
\begin{itemize}
\item
  \eqref{eq:E*} when $f$ is taken to be arithmetic average
  (which is implemented in Algorithm~\ref{alg:BH}),
\item
  and \eqref{eq:E*-i} when $f$ is taken to be product
  (which is implemented in Algorithm~\ref{alg:BH-i}).
\end{itemize}

When the $J$ in \eqref{eq:J} is allowed not to be a singleton and the p-values are replaced by e-values,
we obtain the possibility of controlling false discovery proportion.
This appears to us an interesting program of research;
the ease of merging e-functions open up new possibilities.
First steps in this directions are done in \citet{Vovk/Wang:arXiv1912}
and (under the assumption of independence) in \citet{Vovk/Wang:arXiv2003}.

\subsubsection*{Empirical Bayes methods}

Several simple but informative models for multiple hypothesis testing
have been proposed in the framework of empirical Bayes methods.
Perhaps the simplest model \citep[Chapter 2]{Efron:2010},
known as the \emph{two-groups model},
is where we are given a sequence of real values $z_1,\dots,z_N$,
each of which is generated either from the null probability density function $f_0$
or from the alternative probability density function $f_1$,
w.r.\ to Lebesgue measure.
Each value is generated from $f_0$ with probability $\pi_0$
and from $f_1$ with probability $\pi_1$,
where $\pi_0+\pi_1=1$.
This gives the overall probability density function $f:=\pi_0 f_0+\pi_1 f_1$.

From the Bayesian point of view, the most relevant value for multiple hypothesis testing
is the conditional probability $\fdr(z):=\pi_0 f_0(z)/f(z)$ that an observed value $z$
has been generated from the null probability density function $f_0$;
it is knows as the \emph{local false discovery rate}.
The most natural e-value in this context is the likelihood ratio $e:=f_1(z)/f_0(z)$,
and the local false discovery rate can be written in the form $\fdr(z)=\pi_0/(\pi_0+\pi_1 e)$.
\citet[Section~5.1]{Efron:2010} refers to the ratio $f_1(z)/f_0(z)$ as ``Bayes factor'';
as discussed in Section~\ref{subsec:Bayes-factors},
in this case the notions of e-values and Bayes factors happen to coincide.

A conventional threshold for reporting ``interesting'' cases $z_i$ is $\fdr(z_i)\le0.2$,
where in practice the true $\fdr(z_i)$ is replaced by its empirical estimate
\citep[Section~5.1]{Efron:2010}.
In terms of the likelihood ratio $e$,
the criterion $\fdr(z_i)\le0.2$ can be rewritten as $e\ge4\pi_0/\pi_1$
\citep[Exercise~5.1]{Efron:2010};
of course, the Bayesian decision depends on the ratio of the prior probabilities
of the two hypotheses.
When $\pi_0\ge0.1$ (which is a common case),
we have $e\ge4\pi_0/\pi_1\ge36$
\citep[(5.9)]{Efron:2010},
and so in large-scale hypothesis testing
we need at least very strong evidence on Jeffreys's scale (Section~\ref{sec:experiments})
to declare a case interesting.

The two-groups model is highly idealized;
e.g., all non-null $z$ are assumed to be coming from the same distribution, $f_1$.
In the empirical Bayesian approach the values $z_1,\dots,z_N$ are assumed to satisfy some independence-type conditions
(e.g., \citet{Storey/Tibshirani:2003} assume what they call weak dependence),
in order to be able to estimate relevant quantities and functions, such as $f$, from the data.
In general, this approach makes different assumptions and arrives at different conclusions
as compared with our approach.

\subsection{Test martingales in statistics}

This paper only scratches the surface of the huge topic of test martingales
and their use in statistics.
Martingales were introduced by \citet{Ville:1939} and popularized by \citet{Doob:1953};
see \citet{Mazliak/Shafer:2009} for their fascinating history,
including their applications in statistics.
Recent research includes exponential line-crossing inequalities \citep{Howard/etal:arXiv1808},
nonparametric confidence sequences \citep{Howard/etal:arXiv1810},
and universal inference \citep{Wasserman/etal:2020}.

\section{History and other classes of calibrators}
\label{app:calibrators}

The question of calibration of p-values into Bayes factors
has a long history in Bayesian statistics.
The idea was first raised by \citet[Section 4.2]{Berger/Delampady:1987}
(who, however, referred to the idea as ``ridiculous'';
since then the idea has been embraced by the Bayesian community).
The class of calibrators \eqref{eq:calibrator}
was proposed in \citet{Vovk:1993} and rediscovered in \citet{Sellke/etal:2001}.
A simple characterization of the class of all calibrators was first obtained in \citet{Shafer/etal:2011}.
A popular Bayesian point of view is that p-values tend to be misleading
and need to be transformed into e-values (in the form of Bayes factors) in order to make sense of them.

Recall that the calibrator  \eqref{eq:reviewer} is a mixture of \eqref{eq:calibrator},
and it is closer to $1/p$ than any of \eqref{eq:calibrator} as  $p\to 0$. 
Of course, \eqref{eq:reviewer} is not the only calibrator that is close to $1/p$.
Since
\[
  \int_0^{e^{-1-\kappa}}
  v^{-1} (-\ln v)^{-1-\kappa}
  \d v
  =
  \frac{1}{\kappa(1+\kappa)^\kappa},
\]
where $\kappa\in(0,\infty)$,
each function
\begin{equation}\label{eq:calibrator-2}
  H_{\kappa}(p)
  :=
  \begin{cases}
    \infty & \text{if $p=0$}\\
    \kappa (1+\kappa)^\kappa p^{-1} (-\ln p)^{-1-\kappa} & \text{if $p\in(0,\exp(-1-\kappa)]$}\\
    0 & \text{if $p\in(\exp(-1-\kappa),1]$}
  \end{cases}
\end{equation}
is a calibrator \citep{Shafer/etal:2011}.
It is instructive to compare \eqref{eq:calibrator-2} with $\kappa:=1$ and \eqref{eq:reviewer};
whereas the former benefits from the extra factor of $2$,
it kicks in only for $p\le\exp(-2)\approx0.135$.

We can generalize the calibrator~\eqref{eq:reviewer}
by replacing the uniform distribution on the interval $[0,1]$
by the distribution with density \eqref{eq:calibrator}.
Replacing $\kappa$ in \eqref{eq:reviewer} by $x$ to avoid clash of notation,
we obtain the calibrator
\[
  F_{\kappa}(p)
  :=
  \int_0^1
  x p^{x-1}
  \kappa x^{\kappa-1}
  \d x
  =
  \frac{\kappa\gamma(1+\kappa,-\ln p)}{p(-\ln p)^{1+\kappa}},
\]
where
\[
  \gamma(a,z)
  :=
  \int_0^z
  t^{a-1} \exp(-t)
  \d t
\]
is one of the incomplete gamma functions \citep[8.2.1]{DLMF}.
For $\kappa:=1$ we have
\[
  \gamma(2,-\ln p)
  =
  1 - p + p\ln p
\]
\citep[8.4.7]{DLMF},
which recovers \eqref{eq:reviewer}.
For other positive values of $\kappa$,
we can see that
\[
  F_{\kappa}(p)
  \sim
  \frac{\kappa\Gamma(1+\kappa)}{p(-\ln p)^{1+\kappa}}
\]
as $p\to0$.
The coefficient in \eqref{eq:calibrator-2} is better,
$\Gamma(1+\kappa)<(1+\kappa)^{\kappa}$ for all $\kappa>0$,
but $F_{\kappa}$ gives an informative e-value for all $p$,
not just for $p\le\exp(-1-\kappa)$.

\section{Merging infinite e-values}
\label{app:infty}

Let us check that, despite the conceptual importance of infinite e-values,
we can dispose of them when discussing e-merging functions.

\begin{proposition}\label{prop:infty}
  For any e-merging function $F$,
  the function $F':[0,\infty]^K\to[0,\infty]$ defined by
  \[
    F'(\mathbf{e})
    :=
    \begin{cases}
      F(\mathbf{e}) & \text{if $\mathbf{e}\in[0,\infty)^K$}\\
      \infty & \text{otherwise}
    \end{cases}
  \]
  is also an e-merging function.
  Moreover, $F'$ dominates $F$.
  Neither e-merging function takes value $\infty$ on $[0,\infty)^K$.
\end{proposition}

\begin{proof}
  If $E_1,\dots,E_K$ are e-variables,
  each of them is finite a.s.;
  therefore,
  \[
    F(E_1,\dots,E_K)
    =
    F'(E_1,\dots,E_K)
    \quad
    \text{a.s.},
  \]
  and $F'$ is an e-merging function whenever $F$ is.

  For the last statement, we will argue indirectly.
  Suppose $F(e_1,\dots,e_K)=\infty$
  for some $e_1,\dots,e_K\in[0,\infty)$.
  Fix such $e_1,\dots,e_K\in[0,\infty)$
  and let $E_k$, $k\in\{1,\dots,K\}$, be independent random variables
  such that $E_k$ takes values in the set $\{0,e_k\}$
  (of cardinality 2 or 1),
  takes value $e_k$ with a positive probability,
  and has expected value at most 1.
  (For the existence of such random variables, see Lemma~\ref{lem:rich} below.)
  Since $\E[F(E_1,\dots,E_K)]=\infty$,
  $F$ is not an e-merging function.
\end{proof}

As we mentioned in Section~\ref{sec:ie},
Proposition~\ref{prop:infty} continues to hold for ie-merging functions.

\section{Atomless probability spaces}
\label{app:atomless}


In several of our definitions, such as those of a calibrator or a merging function,
we have a universal quantifier over probability spaces.
Fixing a probability space in those definitions,
we may obtain wider notions.
More generally, in this appendix we will be interested in dependence of our notions
on a chosen statistical model.
We start our discussion from a well-known lemma that we have already used on a few occasions.
(Despite being well-known, the full lemma is rarely stated explicitly;
we could not find a convenient reference in literature.)
Remember that a probability space $(\Omega,\AAA,Q)$ is \emph{atomless}
if it has no \emph{atoms}, i.e., sets $A\in\AAA$
such that $P(A)>0$ and $P(B)\in\{0,P(A)\}$ for any $B\in\AAA$ such that $B\subseteq A$.

\begin{lemma}\label{lem:rich}
  The following three statements are equivalent for any probability space $(\Omega,\AAA,Q)$:
  \begin{enumerate}[(i)]
  \item
    $(\Omega,\AAA,Q)$ is atomless;
  \item
    there is a random variable on $(\Omega,\AAA,Q)$ that is uniformly distributed on $[0,1]$;
  \item
    for any Polish space $S$ and any probability measure $R$ on $S$,
    there is a random element on $(\Omega,\AAA,Q)$ with values in $S$
    that is distributed as $R$.
  \end{enumerate}
\end{lemma}

Typical examples of a Polish space in item (iii) that are useful for us in this paper
are $\R^K$ and finite sets.

\begin{proof}
  The equivalence between (i) and (ii) is stated in \citet[Proposition~A.27]{Follmer/Schied:2011}.
  It remains to prove that (ii) implies (iii).
  According to Kuratowski's isomorphism theorem \citep[Theorem 15.6]{Kechris:1995},
  $S$ is Borel isomorphic to $\R$, $\N$, or a finite set
  (the last two equipped with the discrete topology).
  The only nontrivial case is where $S$ is Borel isomorphic to $\R$,
  in which case we can assume $S=\R$.
  It remains to apply \citet[Proposition~A.27]{Follmer/Schied:2011} again.
\end{proof}

If $(\Omega,\AAA)$ is a measurable space and $\QQQ$ is a collection of probability measures on $(\Omega,\AAA)$,
we refer to $(\Omega,\AAA,\QQQ)$ as a \emph{statistical model}. 
We say that it is \emph{rich}
if there exists a random variable on $(\Omega,\AAA)$ that is uniformly distributed on $[0,1]$
under any $Q\in\QQQ$.

\begin{remark}\label{rem:U}
  Intuitively, any statistical model $(\Omega,\AAA,\QQQ)$ can be made rich
  by complementing it with a random number generator
  producing a uniform random value in $[0,1]$:
  we replace $\Omega$ by $\Omega\times[0,1]$,
  $\AAA$ by $\AAA\times\UUU$,
  and each $Q\in\QQQ$ by $Q\times U$,
  where $([0,1],\UUU,U)$ is the standard measurable space $[0,1]$ equipped with the uniform probability measure $U$.
  If $\QQQ=\{Q\}$ contains a single probability measure $Q$,
  being rich is equivalent to being atomless
  (by Lemma~\ref{lem:rich}).
\end{remark}

For a statistical model $(\Omega,\AAA,\QQQ)$,
an \emph{e-variable} is a random variable $E:\Omega\to[0,\infty]$ satisfying
\[
  \sup_{Q\in\QQQ}
  \E^Q[E]
  \le
  1.
\]
(as in Section~\ref{sec:multiple}).
As before,
the values taken by e-variables are \emph{e-values},
and the set of e-variables is denoted by $\EEE_{\QQQ}$.

An \emph{e-merging function for $(\Omega,\AAA,\QQQ)$}
is an increasing Borel function $F:[0,\infty]^K\to[0,\infty]$
such that, 
for all $E_1,\dots,E_K$,
\begin{equation*} 
  (E_1,\dots,E_K)\in\EEE_{\QQQ}^K
  \Longrightarrow
  F(E_1,\dots,E_K) \in \EEE_{\QQQ}.
\end{equation*}
This definition requires that $K$ e-values for $(\Omega,\AAA,\QQQ)$
be transformed into an e-value for $(\Omega,\AAA,\QQQ)$.
Without loss of generality (as in Appendix~\ref{app:infty}),
we replace $[0,\infty]$ by $[0,\infty)$.

\begin{proposition}\label{prop:model-space}
  Let $F: [0,\infty)^K \to [0,\infty)$ be an increasing  Borel function.
  The following statements are equivalent:
  \begin{enumerate}[(i)]
  \item $F$ is an e-merging function for some rich statistical model;
  \item $F$ is an e-merging function for all statistical models;
  \item $F$ is an e-merging function.
  \end{enumerate} 
\end{proposition}

\begin{proof}
  Let us first check that, for any two rich statistical models
  $(\Omega,\AAA,\QQQ)$ and $(\Omega',\AAA',\QQQ')$,
  we always have
  \begin{equation}\label{eq:equal}
    \sup
    \left\{
      \E^Q[F(\mathbf{E})]
      \mid
      Q\in\QQQ, \enspace \mathbf{E}\in\EEE_{\QQQ}^K
    \right\}
    =
    \sup
    \left\{
      \E^{Q'}[F(\mathbf{E}')]
      \mid
      Q'\in\QQQ', \enspace \mathbf{E}'\in\EEE_{\QQQ'}^K
    \right\}.
  \end{equation}
  Suppose
  \[
    \sup
    \left\{
      \E^Q[F(\mathbf{E})] \mid Q\in\QQQ, \enspace \mathbf{E}\in\EEE_{\QQQ}^K
    \right\}
    >
    c
  \]
  for some constant $c$.
  Then there exist $\mathbf{E}\in\EEE^K_{\QQQ}$ and $Q\in\QQQ$ such that $\E^Q[F(\mathbf{E})] > c$.
  Take a random vector $\mathbf{E}'=(E_1',\dots,E_K')$ on $(\Omega',\AAA')$
  such that $\mathbf{E}'$ is distributed under each $Q'\in\QQQ'$ identically to the distribution of $\mathbf{E}$ under $Q$.
  This is possible as $\QQQ'$ is rich
  (by Lemma~\ref{lem:rich} applied to the probability space $([0,1],\UUU,U)$, $U$ being the uniform probability measure).
  By construction, $\mathbf{E}'\in\EEE_{\QQQ'}^K$ and $E^{Q'}[F(\mathbf{E}')] > c$ for all $Q'\in\QQQ'$.
  This shows
  \[
    \sup
    \left\{
      \E^Q[F(\mathbf{E})]
      \mid
      Q\in\QQQ, \enspace \mathbf{E}\in\EEE_{\QQQ}^K
    \right\}
    \le
    \sup
    \left\{
      \E^{Q'}[F(\mathbf{E}')]
      \mid
      Q'\in\QQQ', \enspace \mathbf{E}'\in\EEE_{\QQQ'}^K
    \right\},
  \]
  and we obtain equality by symmetry.

  The implications $\text{(ii)}\Rightarrow\text{(iii)}$ and $\text{(iii)}\Rightarrow\text{(i)}$ are obvious
  (remember that, by definition an e-merging function
  is an e-merging function for all singleton statistical models).
  To check $\text{(i)}\Rightarrow\text{(ii)}$, suppose $F$ is an e-merging function for some rich statistical model.
  Consider any statistical model.
  Its product with the uniform probability measure on $[0,1]$ will be a rich statistical model
  (cf.\ Remark \ref{rem:U}).
  It follows from \eqref{eq:equal} that $F$ will be an e-merging function for the product.
  Therefore, it will be an e-merging function for the original statistical model.
\end{proof}

\begin{remark}
  The assumption of being rich is essential
  in item (i) of Proposition~\ref{prop:model-space}.
  For instance, if we take $\QQQ := \{\delta_\omega\mid \omega\in\Omega\}$,
  where $\delta_{\omega}$ is the point-mass at $\omega$,
  then $\EEE_{\QQQ}$ is the set of all random variables taking values in $[0,1]$.
  In this case, the maximum of e-variables is still an e-variable,
  but the maximum function is not a valid e-merging function
  as seen from Theorem~\ref{thm:iff}.
\end{remark}

An \emph{ie-merging function for $(\Omega,\AAA,\QQQ)$}
is an increasing Borel function $F:[0,\infty]^K\to[0,\infty]$
such that, for all $E_1,\dots,E_K\in\EEE_{\QQQ}$
that are independent under any $Q\in\QQQ$,
we have $F(E_1,\dots,E_K) \in \EEE_{\QQQ}$.
The proof of Proposition~\ref{prop:model-space} also works
for ie-merging functions.

\begin{proposition}\label{prop:model-space-i}
  Proposition~\ref{prop:model-space} remains true
  if all entries of ``e-merging function'' are replaced by ``ie-merging function''.
\end{proposition}

\begin{proof}
  The changes to the proof of Proposition~\ref{prop:model-space} are minimal.
  In \eqref{eq:equal}, the components of $\mathbf{E}$ and $\mathbf{E}'$ should be assumed to be independent
  under any probability measure in $\QQQ$ and $\QQQ'$, respectively.
  The components of the vector $\mathbf{E}'$ constructed from $\mathbf{E}$ and $Q$ will be independent
  under any $Q'\in\QQQ'$.
\end{proof}

Proposition~\ref{prop:model-space} shows that in the definition of an e-merging function
it suffices to require that \eqref{eq:def-e} hold for a fixed atomless probability space $(\Omega,\AAA,Q)$.
Proposition~\ref{prop:model-space-i} extends this observation to the definition of an ie-merging function.

We can state similar propositions in the case of calibrators.
A \emph{p-variable for a statistical model $(\Omega,\AAA,\QQQ)$}
is a random variable $P:\Omega\to[0,\infty)$ satisfying
\[
  \forall\epsilon\in(0,1) \:
  \forall Q\in\QQQ:
  Q(P\le\epsilon)
  \le
  \epsilon.
\]
The set of p-variables for $(\Omega,\AAA,\QQQ)$ is denoted by $\PPP_{\QQQ}$.
A decreasing function $f:[0,1]\to[0,\infty]$ is a \emph{calibrator for $(\Omega,\AAA,\QQQ)$}
if, for any p-variable $P\in\PPP_{\QQQ}$, $f(P)\in\EEE_{\QQQ}$.

\begin{proposition}\label{prop:model-calibrator}
  Let $f: [0,1] \to [0,\infty]$ be a decreasing  Borel function.
  The following statements are equivalent:
  \begin{enumerate}[(i)]
  \item $f$ is a calibrator for some rich statistical model;
  \item $f$ is a calibrator for all statistical models;
  \item $f$ is a calibrator.
  \end{enumerate} 
\end{proposition}

We refrain from stating the obvious analogue of Proposition~\ref{prop:model-calibrator}
for e-to-p calibrators.

\section{Domination structure of the class of e-merging functions}
\label{app:domination}

In this appendix we completely describe the domination structure of the symmetric e-merging functions,
showing that \eqref{eq:convex} is the minimal complete class of symmetric e-merging functions.
We start, however, with establishing some fundamental facts about e-merging functions.

First, we note that for an increasing Borel function $F:[0,\infty)^K\to[0,\infty]$,
its upper semicontinuous version $F^*$ is given by
\begin{align}\label{eq:usc}
  F^*(\mathbf{e})
  =
  \lim_{\epsilon\downarrow 0}
  F(\mathbf{e} + \epsilon\mathbf 1),
  \quad
  \mathbf{e}\in[0,\infty)^K;
\end{align}
remember that $\mathbf{1}:=(1,\dots,1)$.
Clearly, $F^*$ is increasing, is upper semicontinuous (by a simple compactness argument),
and satisfies $F^*\ge F$.

On the other hand, for an upper semicontinuous (and so automatically Borel) function $F:[0,\infty)^K\to[0,\infty]$,
its increasing version $\widetilde F$ is given by
\begin{equation}\label{eq:inc}
  \widetilde F (\mathbf{e})
  =
  \sup_{\mathbf{e}' \le \mathbf{e}} F(\mathbf{e}'),
  \quad
  \mathbf{e}\in[0,\infty)^K,
\end{equation}
where $\le$ is component-wise inequality.
Clearly, $\widetilde F$ is increasing, upper semicontinuous, and $\widetilde F \ge F$.
Notice that the supremum in \eqref{eq:inc} is attained
(as the supremum of an upper semicontinuous function on a compact set),
and so we can replace $\sup$ by $\max$.

\begin{proposition}\label{prop:usc}
  If $F$ is an e-merging function, then its upper semicontinuous version $F^*$ in \eqref{eq:usc} is also an e-merging function.
\end{proposition}
\begin{proof}
  Take $\mathbf{E}\in\EEE^K_Q$.
  For every rational $\epsilon \in (0,1)$, let $A_\epsilon$ be an event independent of $\mathbf{E}$ with $Q(A_\epsilon)=1-\epsilon$,
  and $\mathbf{E}_{\epsilon} = (\mathbf{E}+\epsilon\mathbf 1)\id_{A_\epsilon}$
  (of course, here we use the convention that $\mathbf{E}_\epsilon=\mathbf{0}:=(0,\dots,0)$
  if the event $A_{\epsilon}$ does not occur).
  For each $\epsilon$, $\E[\mathbf{E}_\epsilon] \le (1-\epsilon) (\mathbf 1+\epsilon \mathbf{1}) \le \mathbf{1}$. 
  Therefore, $\mathbf{E}_\epsilon \in \EEE^K_Q$  and hence
  \begin{align*}
    1
    \ge
    \E[F(\mathbf{E}_\epsilon)]
    =
    (1-\epsilon)
    \E\left[F(\mathbf{E}+\epsilon\mathbf{1})\right]
    +
    \epsilon F(\mathbf{0}),
  \end{align*}
  which implies
  \[
    \E
    \left[
      F(\mathbf{E}+\epsilon\mathbf{1})
    \right]
    \le
    \frac{1 - \epsilon F(\mathbf{0})}{1-\epsilon}.
  \]
  Fatou's lemma yields
  \begin{align*}
    \E[F^*(\mathbf{E})]
    =
    \E
    \left[
      \lim_{\epsilon \downarrow 0}
      F(\mathbf{E}+\epsilon\mathbf{1})
    \right]
    \le
    \lim_{\epsilon \downarrow 0}
    \E
    \left[
      F(\mathbf{E}+\epsilon\mathbf{1})
    \right]
    \le
    \lim_{\epsilon \downarrow 0}
    \frac{1 - \epsilon F(\mathbf 0)}{1-\epsilon}
    =
    1.
  \end{align*}
  Therefore, $F^*$ is an e-merging function. 
\end{proof}

\begin{corollary}\label{cor:usc}
  An admissible e-merging function is always upper semicontinuous. 
\end{corollary}
\begin{proof} 
  Let $F$ be an admissible e-merging function. 
  Using Proposition \ref{prop:usc}, we obtain that $F^*\ge F$ is an e-merging function.
  Admissibility of $F$ forces $F=F^*$,
  implying that $F$ is upper semicontinuous.
\end{proof}

\begin{proposition}\label{prop:inc}
  If $F:[0,\infty)^K\to[0,\infty]$ is an upper semicontinuous function
  satisfying $\E[F(\mathbf{E})]\le 1$ for all $\mathbf{E}\in\EEE^K_Q$,
  then its increasing version $\widetilde F$ in \eqref{eq:inc} is an e-merging function.
\end{proposition}

\begin{proof}
  Take any $\mathbf{E}\in \mathcal{E}_Q^K$ supported in $[0,M]^K$ for some $M>0$.
  Define
  \[
    u(\mathbf{x},\mathbf{y})
    :=
    F(\mathbf{y})
    \text{ and }
    D
    :=
    \left\{
      (\mathbf{x},\mathbf{y})\in[0,M]^{K}\times[0,M]^{K} \mid \mathbf{y}\le\mathbf{x}
    \right\};
  \]
  as a closed subset of a compact set, $D$ is compact.
  Since $F$ is upper semi-continuous,
  the sets
  \[
    U_c := \{(\mathbf{x},\mathbf{y})\in D \mid F(\mathbf{y})\ge c\}
    \text{ and }
    U_c(\mathbf{x}) := \{\mathbf{y} \mid (\mathbf{x},\mathbf{y}) \in U_c\}
  \]
  are all compact (and therefore, Borel).
  Moreover, for each compact subset $\mathcal{K}$ of $[0,M]^K$, 
  the set
  \[
    \left\{
      \mathbf{x} \in [0,M]^K \mid \exists\mathbf{y}: (\mathbf{x},\mathbf{y}) \in U_c \;\&\;\mathbf{y}\in\mathcal{K}
    \right\}
  \]
  is compact (and therefore, Borel).
  These conditions justify the use of Theorem~4.1 of \citet{Rieder:1978},
  which gives the existence of a Borel function $g:[0,M]^K \to [0,M]^K$
  such that $F(g(\mathbf e)) = \widetilde F (\mathbf e )$ and $g(\mathbf e)\le \mathbf e$ for each $\mathbf e\in [0,M]^K$.
  Hence, $g(\mathbf{E}) \in \mathcal{E}_Q^K$, and we have
  \[
    \E[\widetilde F (\mathbf{E})]
    =
    \E[F(g(\mathbf{E}))]
    \le
    1.
  \]
  An unbounded $\mathbf{E} \in \mathcal{E}_Q^K$ can be approximated
  by an increasing sequence of bounded random vectors in $\mathcal E_Q^K$,
  and the monotone convergence theorem implies $\E[\widetilde F (\mathbf{E})] \le 1$.
\end{proof}

\begin{proposition}\label{prop:gap}
  An admissible e-merging function is not strictly dominated by any Borel function $G$
  satisfying $\E[G(\mathbf{E})]\le 1$ for all $\mathbf{E}\in\EEE^K_Q$.
\end{proposition}

\begin{proof}
  Suppose that an admissible e-merging function $F$ is strictly dominated
  by a Borel function $G$ satisfying $\E[G(\mathbf{E})]\le 1$ for all $\mathbf{E}\in\EEE^K_Q$.
  Take a point $\mathbf{e} \in [0,\infty)^K$ such that $G(\mathbf{e})>F(\mathbf{e})$.
  Define a function $H$ by $H(\mathbf{e}) := G(\mathbf{e})$ and $H:=F$ elsewhere.
  By Corollary \ref{cor:usc}, we know that $F$ is upper semicontinuous, and so is $H$ by construction.
  Clearly, $\E[H(\mathbf{E})]\le \E[G(\mathbf{E})] \le 1$ for all $\mathbf{E}\in\EEE^K_Q$.
  Using Proposition \ref{prop:inc}, we obtain that $\widetilde H$ is an e-merging function.
  It remains to notice that $\widetilde H$ strictly dominates $F$.
\end{proof}

\begin{proposition}\label{prop:dominated}
  Any e-merging function is dominated by an admissible e-merging function.
\end{proposition}

\begin{proof}
  Let $R$ be any probability measure with positive density on $[0,\infty)$ with mean $1$.
  Fix an e-merging function $F$.
  By definition, $\int F\d R^K \le 1$,
  and such an inequality holds for any e-merging function.
  Set $F_0:=F$ and let
  \begin{equation}\label{eq:c}
    c_i
    :=
    \sup_{G:G\ge F_{i-1}}
    \int  G \d R^K \le 1,
  \end{equation}
  where $i:=1$ and $G$ ranges over all  e-merging functions dominating $F_{i-1}$.
  Let $F_i$ be an e-merging function satisfying
  \begin{equation}\label{eq:f}
    F_i\ge F_{i-1}
    \quad\text{ and }\quad
    \int F_i\d R^K
    \ge c_{i}-2^{-i},
  \end{equation}
  where $i:=1$.
  Continue setting \eqref{eq:c} and choosing $F_i$ to satisfy \eqref{eq:f}
  for $i=2,3,\dots$.
  Set $G:=\lim_{i\to\infty}F_i$.
  It is clear that $G$ is an e-merging (by the monotone convergence theorem) function dominating $F$
  and that $\int G\d R = \int H\d R$
  for any  e-merging function $H$ dominating $G$.

  By Proposition \ref{prop:usc}, the upper semicontinuous version $G^*$ of $G$ is also an e-merging function.  
  Let us check that $G^*$ is admissible.
  Suppose that there exists an e-merging function $H$ such that $H\ge G^*$ and $H\ne G^*$.
  Fix such an $H$ and an $\mathbf{e}\in [0,\infty)^K$ satisfying $H(\mathbf{e})>G^*(\mathbf{e})$. 
  Since $G^*$ is upper semicontinuous and $H$ is increasing,
  there exists $\epsilon>0$ such that $H>G^*$ on the hypercube
  $[\mathbf{e},\mathbf{e}+\epsilon\mathbf{1}]\subseteq[0,\infty)^K$,
  which has a positive $R^K$-measure.
  This gives
  \[
    \int G\d R^K\le \int G^*\d R^K<\int H\d R^K,
  \]
  a contradiction.
\end{proof}

The key component of the statement of completeness of \eqref{eq:convex} is the following proposition.

\begin{proposition}\label{prop:M-0}
  Suppose that $F$ is a symmetric e-merging function satisfying $F(\mathbf{0})=0$.
  Then $F$ is admissible if and only if it is the arithmetic mean.
\end{proposition}

\begin{proof}
  For the ``if'' statement, see Proposition \ref{prop:IPK}.
  Next we show the ``only if'' statement.
  Let $F$ be an admissible symmetric e-merging function with $F(\mathbf{0})=0$.
  As always, all expectations $\E$ below are with respect to $Q$.

  Suppose for the purpose of contradiction that there exists
  $(e_1,\dots,e_K)\in [0,\infty)^K$ such that
  \begin{equation*}
    F(e_1,\dots,e_K)
    >
    \frac{e_1+\dots+e_K}{K}
    \in
    [0,1)
  \end{equation*}
  (the case ``${}\in[1,\infty)$'' is excluded by Proposition~\ref{prop:M}).
  We use the same notation as in the proof of Proposition~\ref{prop:M}.
  Since $F(\mathbf{0})=0$, we know that $b>a>0$.
  Let $\delta := (b-a)/(1-a)>0$,
  and define $G:[0,\infty)^K\to[0,\infty]$ by $G(\mathbf{0}):=\delta$ and $G:=F$ otherwise.
  It suffices to show that $\E[G(\mathbf{E})]\le 1$ for all $\mathbf{E}\in\EEE^K_Q$;
  by Proposition~\ref{prop:gap} this will contradict the admissibility of $F$.
    
  Since $F$ is an e-merging function,
  for any random vector $(E_1,\dots,E_K)$ taking values in $[0,\infty)^K$
  and any non-null event $B$ independent of $(E_1,\dots,E_K)$ and $(D_1,\dots,D_K)$
  we have the implication:
  if
  \[
    (E_1,\dots,E_K)\id_B + (D_1,\dots,D_K)\id_{B^\complement} \in \EEE_Q^K,
  \]
  then 
  \[
    \E[F((E_1,\dots,E_K)\id_B + (D_1,\dots,D_K)\id_{B^\complement})]
    \le
    1.
  \]
  Write $\beta:=Q(B)$.
  The above statement shows that if
  \[
    \beta \bigvee_{k=1}^K\E[E_k] + (1-\beta) a \le 1,
  \]
  or equivalently,
  \begin{equation}\label{eq:antecedent}
    \bigvee_{k=1}^K\E[E_k]
    \le
    \frac{1-(1-\beta)a}{\beta},
  \end{equation}
  then
  \[
    \beta\E[F(E_1,\dots,E_K)]+ (1-\beta) b \le 1,
  \]
  or equivalently,
  \begin{equation}\label{eq:consequent}
    \E[F(E_1,\dots,E_K)]
    \le
    \frac{1-(1-\beta)b }{\beta}. 
  \end{equation}
  
  Next, take an arbitrary random vector $(E_1,\dots,E_K)$ such that 
  \begin{equation}\label{eq:conj5} 
    Q((E_1,\dots,E_K)\in  [0,\infty)^K\setminus \{\mathbf{0}\}) = 1.
  \end{equation}
  Further, take an arbitrary non-null event $C$ independent of $(E_1,\dots,E_K)$ such that 
  \begin{equation}\label{eq:conj6}
    \bigvee_{k=1}^K\E[E_k]
    \le
    \frac{1}{Q(C)},
  \end{equation}
  which implies $(E_1,\dots,E_K)\id_C \in \EEE_Q^K$.
  We will show that
  \[
    \E[G((E_1,\dots,E_K)\id_C)]\le 1.
  \]
  Write $\lambda := Q(C)$ and choose $\beta\in(0,1]$ such that $\beta/(1-(1-\beta)a)=\lambda$.
  From $\bigvee_{k=1}^K\E[E_k]\le 1/\lambda$ we obtain \eqref{eq:antecedent},
  which implies \eqref{eq:consequent}.
  Using \eqref{eq:consequent}, we have
  \begin{align*}
    \E&[G((E_1,\dots,E_K)\id_C)]\\
    &=
    \lambda \E[F(E_1,\dots,E_K)] + (1-\lambda)\delta \\
    &=
    \lambda \E[F(E_1,\dots,E_K)] + (1-\lambda)\frac{b-a}{1-a} \\
    &\le
    \frac{\beta}{1-(1-\beta)a} \frac{1-(1-\beta)b}{\beta} + \left(1-\frac{\beta}{1-(1-\beta)a}\right) \frac{b-a}{1-a}
    =
    1.
  \end{align*}

  Finally, we note that for any $\mathbf{E}\in\EEE_Q^K$,
  if $Q(\mathbf{E} = \mathbf{0})=0$, then $\E[G(\mathbf{E})] = \E[F(\mathbf{E})] \le 1$.
  If $Q(\mathbf{E} = \mathbf{0})>0$,
  then $\mathbf{E}$ is distributed as $(E_1,\dots,E_K)\id_C$ for some event $C$
  and $(E_1,\dots,E_K)$ satisfying \eqref{eq:conj5}--\eqref{eq:conj6}.
  In either case, we have $\E[G(\mathbf{E})]\le1$.
\end{proof}

Finally, we are able to prove Theorem \ref{thm:iff} based on  Proposition~\ref{prop:M-0}.

\begin{proof}[Proof of Theorem \ref{thm:iff}]
  In view of Proposition~\ref{prop:dominated},
  it suffices to check the characterization of admissibility.
  The ``if'' statement follows from Proposition \ref{prop:IPK}.
  We next show the ``only if'' statement;
  let $F$ be admissible.
  If $F(\mathbf{0})\ge 1$, then $F\ge1$.
  The fact that $F$ is an e-merging function further forces $F=1$.
  Next, assume $F(\mathbf{0})\in[0,1)$ and let $\lambda:=F(\mathbf{0})$.
  Define another function $G:[0,\infty)^K\to[0,\infty)$ by
  \[
    G(\mathbf{e})
    :=
    \frac{F(\mathbf{e})-\lambda}{1-\lambda}.
  \]
  It is easy to see that $G$ is a  symmetric and admissible e-merging function satisfying $G(\mathbf{0})=0$.
  Therefore, using Proposition~\ref{prop:M-0}, we have $G=M_K$.
  The statement of the theorem follows.
\end{proof}

\section{A maximin view of merging}
\label{app:minimax}

This appendix is inspired by the comments by the Associate Editor
of the journal version of this paper.

\subsection{Informal maximin view}

This section is high-level and informal;
in it (and in this appendix in general) we will only discuss the case of e-merging.

In this paper, two particularly important sources of e-merging functions are:
\begin{itemize}
\item
  $\mathcal{F}_0$, the class of all increasing Borel functions $F:[0,\infty)^K \to [0,\infty)$;
\item
  $\mathcal{F}_S$, the class of all symmetric functions in $\mathcal{F}_0$.
\end{itemize}
In the rest of this appendix, $\FFF$ will stand for either $\FFF_0$ or $\FFF_S$.

Let $(\Omega,\AAA,Q)$ be an atomless probability space (cf.\ Appendix~\ref{app:atomless}).
For $\mathbf{E}\in\mathcal{E}^K_Q$,
let $\FFF_{\mathbf{E}}$ be the set of all functions $F\in\FFF$ such that
\[
  \E[F(\mathbf{E})] \le 1;
\]
intuitively, these are $\mathbf E$-specific e-merging functions.

We are looking for suitable e-merging functions to use,
which can be interpreted as the problem of finding the ``best elements'',
in some sense, of
$
  \cap (\FFF_{\mathbf{E}}: \mathbf{E}\in\mathcal{E}^K_Q)
$.
More generally,
we could specify a class $\mathcal M$ of joint models of $K$ e-variables, 
and be interested in
\begin{equation}\label{eq:maximin-crude}
  \max
  \left(
    \bigcap_{\mathbf{E}\in\mathcal{M}}
    \mathcal{F}_{\mathbf{E}}
  \right),
\end{equation}
where $\max(\cdot)$ gives the best element(s) of a set, in some sense.
In the case of admissibility, it will be literally the set of maximal elements,
but it can also be the element essentially or weakly dominating all other elements
if it exists.

The problem \eqref{eq:maximin-crude} can be said to be a maximin problem,
since the natural interpretation of $\cap$ (preceded by $\max$) is minimum.
For the e-merging and ie-merging functions, the informal problems are
\[
  \max
  \left(
    \bigcap_{\mathbf{E}\in\mathcal{E}^K_Q}
    \mathcal{F}_{\mathbf{E}}
  \right)
  \qquad\text{and}\qquad
  \max
  \left(
    \bigcap_{\mathbf{E}\in i\mathcal{E}^K_Q}
    \mathcal{F}_{\mathbf{E}}
  \right).
\]
In the rest of this appendix,
$\MMM$ will stand for either $\EEE^K_Q$ or $i\EEE^K_Q$.


\subsection{Formal maximin}
\label{subsec:maximin}

Our results about essential and weak domination,
namely Propositions~\ref{prop:M} and~\ref{prop:M-i},
have interesting connections with the maximin problem
\begin{equation}\label{eq:maximin}
  \sup_{F\in\FFF}
  \min_{\mathbf{E}\in\MMM}
  F(\mathbf{e})
  \id_{\{\E[F(\mathbf{E})]\le 1\}}
\end{equation}
for a fixed $\mathbf{e}\in [0,\infty)^K$.
The value~\eqref{eq:maximin} is the supremum of $F(\mathbf{e})$
over all e-merging functions (if $\MMM=\EEE^K_Q$)
or over all ie-merging functions (if $\MMM=i\EEE^K_Q$).
Intuitively, this corresponds to an overoptimistic way of merging e-values
choosing the best merging function in hindsight
(which makes \eqref{eq:maximin} somewhat similar to the VS bound).
Notice that the minimum in \eqref{eq:maximin}
(either $F(\mathbf{e})$ or 0)
is indeed attained.

Fix $\mathbf{e}\in [0,\infty)^K$.
Propositions~\ref{prop:M} and~\ref{prop:M-i} show that
\begin{align*}
  \max_{F\in\mathcal{F}_S}
  \min_{\mathbf{E}\in\mathcal{E}^K_Q}
  F(\mathbf{e})
  \id_{\{\E[F(\mathbf E)]\le 1\}}
  &=
  M_K(\mathbf{e}) \vee 1,\\
  \max_{F\in\mathcal{F}_S}
  \min_{\mathbf{E}\in i\mathcal {E}^K_Q}
  F(\mathbf e)
  \id_{\{\E[F(\mathbf{E})]\le 1\}}
  &=
  P_K(\mathbf{e}\vee\mathbf{1})\\
  &=
  \max_{F\in\mathcal{F}_0}
  \min_{\mathbf{E}\in i\mathcal {E}^K_Q}
  F(\mathbf e)
  \id_{\{\E[F(\mathbf{E})]\le 1\}},
\end{align*}
where $M_K$ is the arithmetic mean and $P_K$ is the product function.
This follows from the maximin problem having a universal optimizer for $\mathbf{e}$ large enough:
$M_K(\mathbf{e})\ge 1$ in the e-merging case and $\mathbf{e}\ge\mathbf{1}$ in the ie-merging case.

Let us check that
\[
  \max_{F\in\FFF_0}
  \min_{\mathbf{E}\in\EEE^K_Q}
  F(\mathbf e)
  \id_{\{\E[F(\mathbf{E})]\le 1\}}
  =
  \max(\mathbf{e})\vee 1.
\]
To show this,
first notice that,
for each $k=1,\dots,K$,
$(e_1,\dots,e_K)\mapsto e_k$ is an e-merging function,
and so is $(e_1,\dots,e_K)\mapsto 1$.
This shows the $\ge$ part of the equality. 
On the other hand, for any function $F\in\FFF$,
if $F(\mathbf{e})>a:=\max(\mathbf e)\vee 1$ for some $\mathbf{e}$,
then by designing a vector $\mathbf{E}$ of e-variables with $Q(\mathbf{E}=\mathbf{e})=1/a$,
we have $\E[F(\mathbf{E})] \ge F(\mathbf{e})/a > 1$,
and hence $F$ is not an e-merging function.
This gives the $\le$ part of the equality.

\subsection{The minimax formulation}

For a fixed $\mathbf{e}\in [0,\infty)^K$, we can also talk about the minimax problem
corresponding to the maximin problem~\eqref{eq:maximin}:
\begin{equation}\label{eq:minimax}
  \inf_{\mathbf{E}\in\MMM}
  \max_{F\in\FFF}
  F(\mathbf{e})
  \id_{\{\E[F(\mathbf{E})]\le 1\}}.
\end{equation}
As usual, we have
\begin{equation}\label{eq:minimax2}
  \inf_{\mathbf{E}\in\MMM}
  \max_{F\in\FFF}
  F(\mathbf{e})
  \id_{\{\E[F(\mathbf{E})]\le 1\}}
  \ge
  \sup_{F\in\FFF}
  \min_{\mathbf{E}\in\MMM}
  F(\mathbf{e})
  \id_{\{\E[F(\mathbf{E})]\le 1\}},
\end{equation}
but the two sides are not always equal.

The minimax problem \eqref{eq:minimax} is usually easy to solve.
We first look at the case $\FFF=\FFF_0$.
Note that for fixed $\mathbf{E}\in\EEE^K_Q$, using the Neyman--Pearson argument, we have
\[
  \max_{F\in\FFF_0}
  F(\mathbf{e})
  \id_{\{\E[F(\mathbf{E})]\le 1\}}
  =
  \frac{1}{Q(\mathbf{E}\ge\mathbf{e})}
  \in
  [1,\infty],
\]
for which a maximizer (typically the unique maximizer) is
\[
  \mathbf{e}'
  \mapsto
  \frac
  {\id_{\{\mathbf{e}'\ge\mathbf{e}\}}}
  {Q(\mathbf{E}\ge\mathbf{e})}.
\]
Therefore, the minimax problem \eqref{eq:minimax} for $\FFF=\FFF_0$ has value
\[
  \inf_{\mathbf{E}\in\MMM}
  \max_{F\in\FFF_0}
  F(\mathbf{e})
  \id_{\{\E[F(\mathbf{E})]\le 1\}}
  =
  \inf_{\mathbf{E}\in\MMM}
  \frac{1}{Q(\mathbf{E}\ge\mathbf{e})}.
\]
Since $\MMM$ is $\mathcal E^K_Q$ or $i\mathcal E^K_Q$,
we can compute this as
\begin{align}
  \min_{\mathbf{E}\in\EEE^K_Q}
  \frac{1}{Q(\mathbf{E}\ge\mathbf{e})}
  &=
  \min_{E\in\EEE_Q}
  \frac{1}{Q(E\ge\max(\mathbf{e}))}
  =
  \max(\mathbf{e})\vee 1,
  \notag\\
  \min_{\mathbf{E}\in i\EEE^K_Q}
  \frac{1}{Q(\mathbf{E}\ge\mathbf{e})}
  &=
  \prod_{k=1}^K
  \min_{E_k\in\EEE_Q}
  \frac{1}{Q(E_k\ge e_k)}
  =
  P_K(\mathbf{e}\vee\mathbf{1}).
\label{eq:minimax3}
\end{align}
In combination with results of Section~\ref{subsec:maximin},
this shows that \eqref{eq:minimax2} holds as an equality in this case.

Next, let us look at the case $\FFF=\FFF_S$.
For $\MMM=i\EEE^K_Q$,
the coincidence of the minimax and maximin follows from the previous results,
so we assume $\MMM=\EEE^K_Q$.
Let 
\[
  A_{\mathbf{e}}
  :=
  \bigcup_{\pi\in\Pi_K}
  \{\mathbf{e}'\in[0,\infty)^K: \mathbf{e}'\ge\mathbf{e}_{\pi}\},
\]
where $\Pi_K$ is the set of all $K$-permutations and $\mathbf{e}_{\pi}:=(e_{\pi(1)},\dots,e_{\pi(K)})$.
Using a Neyman--Pearson argument again, we have, for a fixed $\mathbf{E}\in\EEE^K_Q$,
\[
  \max_{F\in\FFF_S}
  F(\mathbf{e})
  \id_{\{\E[F(\mathbf{E})]\le 1\}}
  =
  \frac
  {1}
  {Q(\mathbf{E}\in A_{\mathbf{e}})},\]
for which a maximizer is
\[
  \mathbf{e}'
  \mapsto
  \frac
  {\id_{\{\mathbf{e}'\in A_{\mathbf{e}}\}}}
  {Q(\mathbf{E}\in A_{\mathbf{e}})}.
\] 
Let $a:=(1/M_K(\mathbf e))\wedge 1$, and the distribution of $\mathbf{E}'$ be given by
\[
  \frac{a}{K!}
  \sum_{\pi\in\Pi_K}
  \delta_{\mathbf e_\pi}
  +
  (1-a)\delta_{\mathbf{0}},
\]
$\delta_{\omega}$ being the point-mass at $\omega$.
It is clear that $\mathbf{E}'\in\EEE^K_Q$ and $Q(\mathbf{E}'\in A_{\mathbf{e}})= a$.
It follows that
\begin{equation}\label{eq:minimax4}
  \inf_{\mathbf{E}\in\EEE^K_Q}
  \frac{1}{Q(\mathbf{E}\in A_{\mathbf{e}})}
  \le
  \frac{1}{Q(\mathbf{E}'\in A_{\mathbf{e}})}
  =
  M_K(\mathbf{e}) \vee 1.
\end{equation}
Hence, by~\eqref{eq:minimax2} and the results of Section~\ref{subsec:maximin},
the values of the maximin and the minimax again coincide
(and the inf in \eqref{eq:minimax4} is actually attained
and so can be replaced by min).

The inequalities \eqref{eq:minimax3} and \eqref{eq:minimax4}
give alternative proofs to the domination statements in Propositions~\ref{prop:M}
and (in the independent case) \ref{prop:M-i},
since our overoptimistic upper bound coincides with a valid e-merging (ie-merging in the independent case) function
when $M_K(\mathbf{e})\ge 1$ (when $\mathbf{e} \ge \mathbf{1}$ in the independent case).
For this, we do not need anything derived in the main paper or in Section~\ref{subsec:maximin} above.
However, the minimax approach (at least in the form presented here)
does not produce the full domination structure of symmetric e-merging functions
as given in Theorem~\ref{thm:iff}.

\section{Cross-merging between e-values and p-values}
\label{app:cross-merging}

In this appendix we will briefly discuss functions performing ``cross-merging'':
either merging several e-values into a p-value or several p-values into an e-value.
Formally, an \emph{e-to-p merging function} is a decreasing Borel function $F:[0,\infty]^K\to[0,1]$
such that $F(E_1,\dots,E_K)$ is a p-variable whenever $E_1,\dots,E_K$ are e-variables,
and a \emph{p-to-e merging function} is a decreasing Borel function $F:[0,1]^K\to[0,\infty]$
such that $F(P_1,\dots,P_K)$ is an e-variable whenever $P_1,\dots,P_K$ are p-variables.
The message of this appendix is that cross-merging can be performed
as composition of pure merging (applying an e-merging function or a p-merging function)
and calibration (either e-to-p calibration or p-to-e calibration);
however, in some important cases (we feel in the vast majority of cases)
pure merging is more efficient, and should be done, in the domain of e-values.

Let us start from e-to-p merging.
Given e-values $e_1,\dots,e_K$,
we can merge them into one e-value by applying the arithmetic mean,
the only essentially admissible symmetric e-merging function (Proposition~\ref{prop:M}),
and then by applying inversion $e\mapsto e^{-1}\wedge1$,
the only admissible e-to-p calibrator (Proposition~\ref{prop:e-to-p}).
This gives us the e-to-p merging function
\begin{equation}\label{eq:e-to-p-merging}
  F(e_1,\dots,e_K)
  :=
  \frac{K}{e_1+\dots+e_K}
  \wedge
  1.
\end{equation}
The following proposition shows that in this way we obtain the optimal symmetric e-to-p merging function.

\begin{proposition}
  The e-to-p merging function \eqref{eq:e-to-p-merging} dominates all symmetric e-to-p merging functions.
\end{proposition}

\begin{proof}
  Suppose that a symmetric e-to-p merging function $G$ satisfies
  $G(\mathbf{e})<F(\mathbf{e})$ for some $\mathbf{e}=(e_1,\dots,e_K)\in[0,\infty)^K$.
  The following arguments are similar to the proof of Proposition~\ref{prop:M}.
  As before, $\Pi_K$ is the set of all permutations on $\{1,\dots,K\}$,
  $\pi$ is randomly and uniformly drawn from $\Pi_K$,
  and $(D_1,\dots,D_K):=(e_{\pi(1)},\dots,e_{\pi(K)})$.
  Further, let $(D'_1,\dots,D'_K) := (D_1,\dots,D_K) 1_A$,
  where $A$ is an event independent of $\pi$ and satisfying $Q(A) = F(\mathbf{e})$.
  For each $k$, we have $\E[D'_k]= F(\mathbf{e}) M_{K}(e_1,\dots,e_K) \le 1$,
  and hence $D'_k\in\EEE_Q$.
  By the symmetry of $G$, we have $Q(G(D'_1,\dots,D'_K) = G(\mathbf{e})) \ge Q(A) = F(\mathbf{e})$,
  and hence
  \[
    Q
    \left(
      G(D'_1,\dots,D'_K)
      \le
      G(\mathbf{e})
    \right)
    \ge
    F(\mathbf{e})
    >
    G(\mathbf{e}).
  \]
  This contradicts $G$ being an e-to-p merging function.
\end{proof}

It is interesting that \eqref{eq:e-to-p-merging}
can also be obtained by composing e-to-p calibration and improper pure p-merging.
Given e-values $e_1,\dots,e_K$ we first transform them into p-values $1/e_1,\dots,1/e_K$
(in this paragraph we allow p-values greater than 1, as in \citet{Vovk/Wang:2019}).
\citet{Wilson:2019} proposed the harmonic mean as a p-merging function.
The composition of these two transformations again gives us the e-to-p merging function \eqref{eq:e-to-p-merging}.
The problem with this argument is that,
as \citet[Wilson's second claim]{Goeman/etal:2019} point out,
Wilson's method is in general not valid
(one obtains a valid method if the harmonic mean is multiplied by $c\ln K$
for $K>2$ and for some constant $c<\exp(1)$,
according to \citet{Vovk/Wang:2019}).
Despite the illegitimate application of the harmonic mean,
the resulting function \eqref{eq:e-to-p-merging} is still a valid e-to-p merging function.
At least in this context,
we can see that e-to-p merging should be done by first pure merging and then e-to-p calibration,
not vice versa (which would result in an extra coefficient of $c\ln K$).

Now suppose we are given p-values $p_1,\dots,p_K$,
and we would like to merge them into one e-value.
Let $\kappa\in(0,1)$.
Applying the calibrator \eqref{eq:calibrator},
we obtain e-values $\kappa p_1^{\kappa-1},\dots,\kappa p_K^{\kappa-1}$,
and since the average of e-values is an e-value,
\begin{equation}\label{eq:p-to-e-merging}
  F(p_1,\dots,p_K)
  :=
  \frac{\kappa}{K}
  \sum_{k=1}^K p_k^{\kappa-1}
\end{equation}
is a p-to-e merging function.

The following proposition will imply that all p-to-e merging functions \eqref{eq:p-to-e-merging} are admissible;
moreover, it will show, in conjunction with Proposition~\ref{prop:p-to-e},
that for any admissible (p-to-e) calibrator $g$, the function
\begin{equation*} 
  M_g(p_1,\ldots,p_K)
  :=
  \frac1K
  \sum_{k=1}^K
  g(p_k)
\end{equation*}
is an admissible p-to-e merging function.  

\begin{proposition}
  If $F:[0,1]^K\to[0,\infty]$ is an upper semicontinuous and decreasing Borel function,
  $\E[F(\mathbf{P})] = 1$ for all $\mathbf{P}\in\PPP_Q^K $ with margins uniform on $[0,1]$,
  and $F=\infty$ on $[0,1]^K\setminus(0,1]^K$,
  then $F$ is an admissible p-to-e merging function.
\end{proposition}

\begin{proof}
  It is obvious
  (cf.\ the proof of Proposition \ref{prop:p-to-e})
  that $F$ is a p-to-e merging function.
  To show that $F$ is admissible,
  consider another p-to-e merging function $G$ such that $G\ge F$.
  For independent $P_1,\dots,P_K$ distributed uniformly on $[0,1]$,
  \[
    1 \ge \E[G(P_1,\dots,P_K)] \ge \E[F(P_1,\dots,P_K)] = 1,
  \]
  forcing $G=F$ almost everywhere on $[0,1]^K$.
  The upper semicontinuity of $F$ and $G$ being decreasing further guarantee that $G=F$ on $(0,1]^K$;
  indeed, if $G(\mathbf{e})>F(\mathbf{e})$ for $\mathbf{e}\in(0,1]^K$,
  there exists $\epsilon>0$ such that $G>H$ on the hypercube
  $[\mathbf{e}-\epsilon\mathbf{1},\mathbf{e}]\subseteq(0,1]^K$,
  which has a positive Lebesgue measure.
  Therefore, $F$ is admissible.
\end{proof}

Let us see how we can obtain \eqref{eq:p-to-e-merging}
reversing the order in which we do calibration and pure merging.
If we first merge the p-values $p_1,\dots,p_K$ by naively (improperly) assuming that their generalized mean
\begin{equation}\label{eq:mean}
  \left(
    \frac1K
    \sum_{k=1}^K p_k^{\kappa-1}
  \right)^{\frac{1}{\kappa-1}}
\end{equation}
is a p-value
and then apply the calibrator \eqref{eq:calibrator},
we will obtain exactly the p-to-e merging function \eqref{eq:p-to-e-merging}.
As shown in \citet[Table 1]{Vovk/Wang:2019},
\eqref{eq:mean} is not a valid p-value in general
(and has to be multiplied by at least $\kappa^{1/(\kappa-1)}$ to get a valid p-value).
This lack of validity, however, does not matter in this context:
the final result \eqref{eq:p-to-e-merging} is still a valid p-to-e merging function.
This shows that, in the  context of p-to-e merging, one should  first perform p-to-e calibration and then pure merging,
not vice versa.

\begin{remark}
  We can generalize \eqref{eq:p-to-e-merging} to
  \[
    F(p_1,\dots,p_K)
    :=
    \sum_{k=1}^K
    \lambda_k
    f_k(p_k),
  \]
  where $\lambda_k\ge0$ sum to 1 and $f_k$ are calibrators.
  It is interesting that any p-to-e merging function is dominated
  by a p-to-e merging function of this form
  \citep[Theorem~4.1]{Vovk/etal:2020}.
  We can see that the classes of e-merging functions,
  e-to-p merging functions, and p-to-e merging functions admit simple explicit representations.
  But the class of p-merging functions is much more complex.
\end{remark}

  \subsection*{Cross-merging under dependency assumptions}

  In the main part of this appendix we have discussed cross-merging for arbitrarily dependent e-values and p-values.
  Cross-merging of independent p-values might be trivial
  (we can take any decreasing Borel $F:[0,1]^K\to[0,\infty]$ satisfying $\int F\d U\le1$,
  $U$ being the uniform probability measure on $[0,1]^K$),
  but cross-merging of independent e-values raises interesting problems.

  Combining the only admissible e-to-p calibrator of Proposition~\ref{prop:e-to-p}
  and the product ie-merging function \eqref{eq:product},
  we obtain the ``ie-to-p'' merging function
  \[
    (e_1,\dots,e_K)
    \mapsto
    \frac{1}{e_1\dots e_K}
    \wedge
    1;
  \]
  this is even an \emph{se-to-p merging function},
  in the sense of mapping any sequential e-values to a p-value.
  However, we can do better:
  by Ville's theorem \citep[p.~100]{Ville:1939},
  \[
    (e_1,\dots,e_K)
    \mapsto
    \min_{k=0,\dots,K}
    \frac{1}{e_1\dots e_k}
    \wedge
    1
  \]
  is also an se-to-p merging function.
  For further information,
  see \citet{Shafer/etal:2011}.

  Strengthening the assumption of $e_1,\dots,e_K$ being sequential to their independence
  opens up new possibilities for their combination;
  cf.\ \citet{Vovk/Wang:arXiv2007}.

\section{Additional experimental results}
\label{app:extra-experiments}

In our experiments in the subsection ``Combining independent e-values and p-values''
in Section~\ref{sec:experiments}
we considered the case where the alternative hypothesis was always true.
In this appendix we will report results of experiments in the situation
where it is true only part of the time.

\begin{figure}
  \begin{center}
    \includegraphics[width=0.6\textwidth]{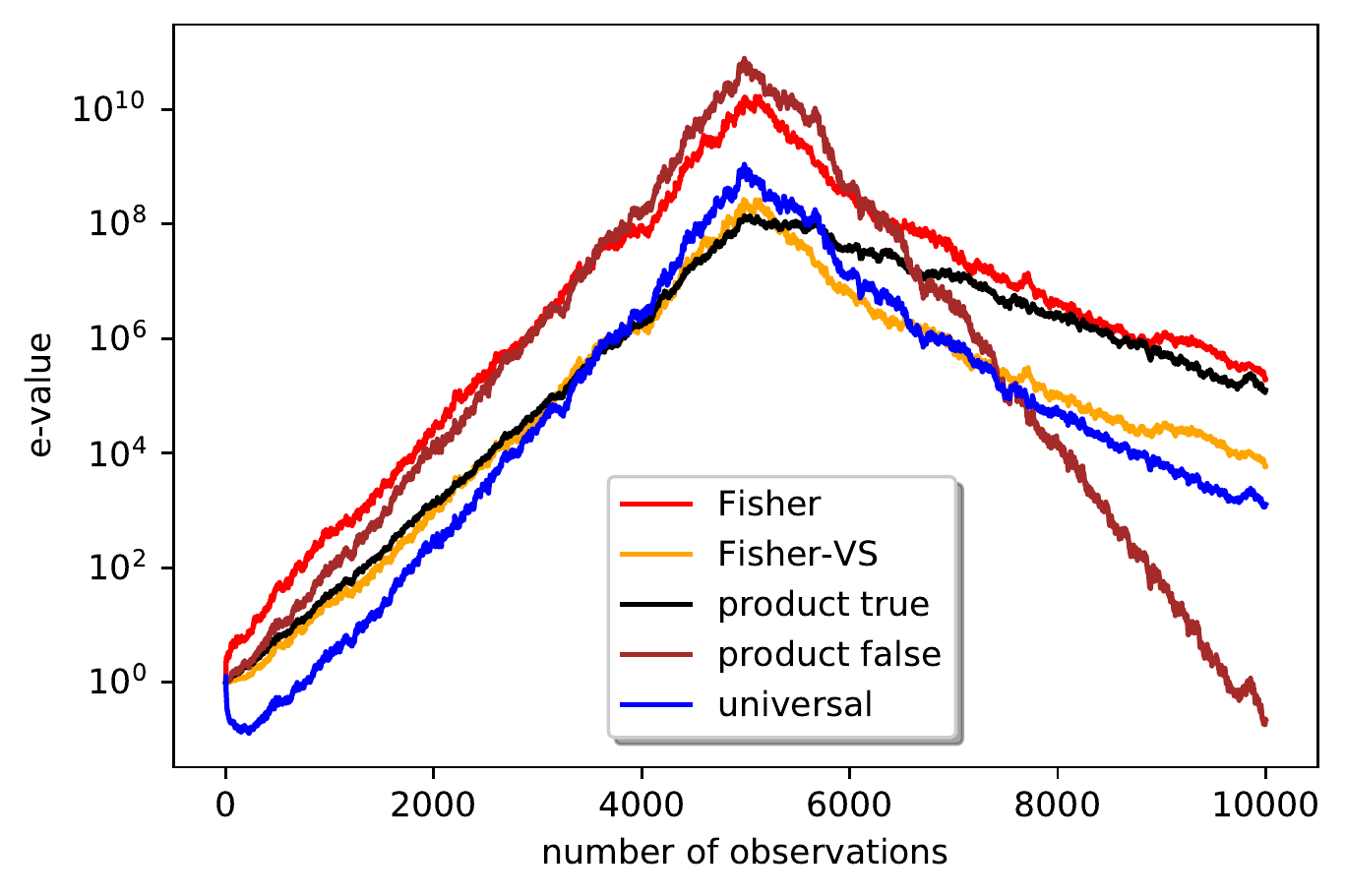}
  \end{center}
  \caption{The analogue of Figure~\ref{fig:combining}
    where the alternative hypothesis is true half of the time
    (details in text).}
  \label{fig:combining_mix}
\end{figure}

Figure~\ref{fig:combining_mix} uses a similar setting to Figure~\ref{fig:combining};
in particular,
the observations are generated from the Gaussian model $N(\mu,1)$,
the null hypothesis is $\mu=0$ and the alternative hypothesis is $\mu=-0.1$.
But now we generate only half (namely, the first half) of the data ($10{,}000$ observations overall) from the alternative distribution,
and the rest from the null distribution.
The e-variable is the likelihood ratio \eqref{eq:data-2} of the ``true'' probability density to the null probability density,
so that we assume it known that half of the observations are generated from the alternative distribution.
The results for \eqref{eq:data-2} are shown in Figure~\ref{fig:combining_mix} as the black line
(all graphs in that figure use the medians over 100 seeds).
Comparing the black line with the red line (representing Fisher's method),
we can see that their final values are approximately the same.
For the comparison to be fairer, we should compare the black line with the orange one
(representing the VS bound for Fisher's method);
the final value for the black line is significantly higher.
Despite the method of multiplication lagging behind Fisher's and the VS bound for it over the first half of the data,
it then catches up with them.

As we said in Section~\ref{subsec:Bayes-factors},
p-values are usually associated with frequentist statistics
while e-values are closely connected to Bayesian statistics.
As discussed in Section~\ref{sec:experiments},
the latter often require stronger assumptions,
which is typical of Bayesian statistics.
This can be illustrated using the two ways of generating data
that we consider in Section~\ref{sec:experiments} and in this appendix so far:
always using $N(-0.1,1)$ or first using $N(-0.1,1)$ and then $N(0,1)$.
Whereas the p-value is always computed using the same formula
(namely, \eqref{eq:p-1}),
the e-value is computed as the likelihood ratio~\eqref{eq:data-1} or the likelihood ratio~\eqref{eq:data-2}.
The fact that more knowledge is assumed in the case of e-values
is further illustrated by the brown line in Figure~\ref{fig:combining_mix},
which is the graph for the product rule that uses the ``wrong'' likelihood ratio \eqref{eq:data-1}
in the case where the alternative hypothesis is true half of the time
(as for the other graphs in that figure).
Over the first half of the data the product rule performs very well
(as in Figure~\ref{fig:combining}),
but then it loses all evidence gathered against the null hypothesis.
Its final value is approximately 1, despite the null hypothesis being false.
The blue line corresponds to the universal test martingale~\eqref{eq:universal}
and does not have this deficiency.

  \section{FACT algorithm}
  \label{app:FACT}

  \begin{algorithm}[bt]
    \caption{FACT (FAst Closed Testing)}
    \label{alg:FACT}
    \begin{algorithmic}[1] 
      \Require
        A sequence of p-values $p_1,\dots,p_K$.
      \State Find a permutation $\pi$ of $\{1,\dots,K\}$ such that $p_{\pi(1)}\le\dots\le p_{\pi(K)}$.
      \State Define the order statistics $p_{(k)}:=p_{\pi(k)}$, $k\in\{1,\dots,K\}$.
      \For{$i=1,\dots,K$}
        \State $P_i:=F(p_{(i)},\dots,p_{(K)})$
      \EndFor
      \For{$k=1,\dots,K$}
        \State $p^*_{\pi(k)}:=F(p_{\pi(k)})$
        \For{$i=k+1,\dots,K$}
          \State $p := F(p_{\pi(k)},p_{(i)},\dots,p_{(K)})$
          \If{$p > p^*_{\pi(k)}$}
            \State $p^*_{\pi(k)} := p$
          \EndIf
        \EndFor
        \For{$i=1,\dots,k$} 
          \If{$P_i > p^*_{\pi(k)}$}
            \State $p^*_{\pi(k)} := P_i$ 
          \EndIf
        \EndFor
      \EndFor
    \end{algorithmic}
  \end{algorithm}

  Algorithm~\ref{alg:FACT} is a generic procedure that turns any p-merging function $F$
  into a function performing multiple hypothesis testing.
  It is equivalent to the closed testing procedure provided the p-merging function $F$ is symmetric
  and monotonically increasing in each (equivalently, any) of its arguments.
  It is a version of Dobriban's [\citeyear{Dobriban:2020}] Algorithm~2.

  \begin{algorithm}[bt]
    \caption{FACT on top of Fisher's method}
    \label{alg:FACT-Fisher}
    \begin{algorithmic}[1] 
      \Require
        A sequence of p-values $p_1,\dots,p_K$.
      \State Find a permutation $\pi$ of $\{1,\dots,K\}$ such that $p_{\pi(1)}\le\dots\le p_{\pi(K)}$.
      \State Define the order statistics $p_{(k)}:=p_{\pi(k)}$, $k\in\{1,\dots,K\}$.
      \State $S_{K+1}:=0$
      \For{$i=K,\dots,1$}
        \State $S_i := S_{i+1} - 2\ln p_{(i)}$
        \State $P_i:=1-F^{\chi^2}_{2(K+1-i)}(S_i)$
      \EndFor
      \For{$k=1,\dots,K$}
        \State $p^*_{\pi(k)}:=p_{\pi(k)}$\label{ln:identity}
        \For{$i=K,\dots,k+1$}
          \State $p := 1-F^{\chi^2}_{2(K+2-i)}(-2\ln p_{\pi(k)} + S_i)$
          \If{$p > p^*_{\pi(k)}$}
            \State $p^*_{\pi(k)} := p$
          \EndIf
        \EndFor
        \For{$i=1,\dots,k$}
          \If{$P_i > p^*_{\pi(k)}$}
            \State $p^*_{\pi(k)} := P_i$
          \EndIf
        \EndFor
      \EndFor
    \end{algorithmic}
  \end{algorithm}

  When specialized to Fisher's combination method,
  Algorithm~\ref{alg:FACT} becomes Algorithm~\ref{alg:FACT-Fisher},
  where $F^{\chi^2}_{n}$ stands for the $\chi^2$ distribution function with $n$ degrees of freedom
  and line~\ref{ln:identity} uses the easy-to-check identity
  \[
    1 - F^{\chi^2}_2(-2\ln p)
    =
    p.
  \]
  Algorithm~\ref{alg:FACT-Fisher} is used in our code \citep{Vovk/Wang:code}
  for producing Figures~\ref{fig:multiple_small} and~\ref{fig:multiple_big}.
\end{document}